\documentclass [oneside, 11pt, reqno] {amsart}

\usepackage[vmargin={2.8cm,3.3cm},hmargin=2.8cm,left=1.2in,right=1.2in]{geometry}

\usepackage [latin1] {inputenc}
\usepackage{amsmath}
\usepackage{amssymb}
\usepackage{amsfonts}
\usepackage{exscale}
\usepackage{graphicx}
\usepackage{setspace}
\usepackage{amsthm}
\usepackage{stmaryrd}

\makeatletter
\renewcommand\paragraph{\@startsection{paragraph}{4}{\z@}%
  {-3.25ex\@plus -1ex \@minus -.2ex}%
  {1.5ex \@plus .2ex}%
  {\normalfont\normalsize\itshape}}
\makeatother

\long\def\symbolfootnote[#1]#2{\begingroup%
\def\thefootnote{\fnsymbol{footnote}}\footnote[#1]{#2}\endgroup} 

\addtocontents{toc}{\protect\thispagestyle{empty}}

\newcounter{assumptions}

\newtheorem{ass}[assumptions]{Assumption}%[section]

\begin{document}

\newtheorem{lem}{Lemma}[section]
\newtheorem{cor}[lem]{Corollary}%[section]
\newtheorem{prop}[lem]{Proposition}%[section]
\newtheorem{thm}[lem]{Theorem}%[section]

\title[Sequential MCMC Methods in Multimodal Settings]{Non-asymptotic Error Bounds for Sequential MCMC Methods in Multimodal Settings}
\author{Nikolaus Schweizer}
\date{March 2012}

\address{Institute for Applied Mathematics, University of Bonn, Endenicher Allee 60, 53115 Bonn}

\email{nschweizer@iam.uni-bonn.de}
\thanks{Financial support of the German Research Foundation (DFG)  through the Hausdorff Center for Mathematics is gratefully acknowledged.}

\date{March 2012}


\begin{abstract}
We prove non-asymptotic error bounds for Sequential MCMC methods in the case of multimodal target distributions. Our bounds depend in an explicit way on upper bounds on relative densities, on constants associated with local mixing properties of the MCMC dynamics, namely, local spectral gaps and local hyperboundedness,  and on the amount of probability mass shifted between effectively disconnected components of the state space.
\end{abstract}

\maketitle

\section{Introduction}

Sequential MCMC methods, see \cite{DDJ05} and the references therein, are a class of stochastic methods for  numerical integration with respect to target probability measures $\mu$ which cannot feasibly be attacked directly with standard MCMC methods due to the presence of multiple well-separated modes. The basic idea is to approximate the target distribution $\mu$ with a sequence of distributions $\mu_0,\ldots, \mu_n$ such that $\mu_n=\mu$ is the actual target distribution and such that $\mu_0$ is easy to sample from. The distributions $(\mu_k)_k$ interpolate between $\mu_0$ and $\mu_n$ in a suitable way and, roughly, the algorithm tries to carry  the good sampling properties of $\mu_0$ over to $\mu_n$.\bigskip

The algorithm constructs a system of $N$ particles which sequentially approximates the measures $\mu_0$ to $\mu_n$. It is initialized with $N$ independent samples from $\mu_0$ and then alternates two types of steps, Importance Sampling Resampling and MCMC: In the Importance Sampling Resampling steps, a cloud of particles approximating $\mu_k$ is transformed into a cloud of particles approximating $\mu_{k+1}$ by randomly duplicating and eliminating particles in a suitable way depending on the relative density between $\mu_{k+1}$ and $\mu_{k}$. In the MCMC steps, particles move independently according to an MCMC dynamics for the current target distribution in order to adjust better to the changed environment.  We focus on a simple Sequential MCMC method with Multinomial Resampling which is a basic instance of the class of algorithms introduced in Del Moral, Doucet and Jasra \cite{DDJ05}.\bigskip

The algorithm is essentially the same as the particle filter of Gordon, Salmond and Smith \cite{GSS93} yet the application - sampling from a fixed target distribution instead of filtering with an exogenous sequence of distributions - is different. The most common way of choosing the sequence $\mu_k$ consists in setting $\mu_k(dx) \sim \exp(-\beta_k H(x))\pi(dx)$ for an increasing sequence $\beta_k$ of (artificial) inverse temperature parameters, see, e.g., Neal \cite{N01}, a reference measure $\pi$ and a Hamiltonian $H$ chosen such that $\mu_n$ is identical to the desired target distribution. Smaller values of $\beta$ surpress differences in $H$ and thus lead to flatter distributions which are easier to sample using MCMC.\bigskip

Our main question is the following: How well does the mean $\eta^N_n(f)$ of an integrand $f$ with respect to the empirical measure $\eta^N_n$ of the particle system approximate the integral of interest $\mu_n(f)$? We address this question by proving non-asymptotic error bounds of the type
\[
\mathbb{E}[(\mu_n(f)-\eta_n^N(f))^2]\leq \frac{C_n(f)}{N},
\]
where $\mathbb{E}$ is the expectation with respect to the randomness in the particle system and $C_n(f)$ is a constant depending on the model parameters and on the function $f$ in an explicit way. More specifically, our results address the following question: Under which conditions does the particle dynamics work well in multimodal settings where conventional MCMC methods are trapped in local modes? We prove non-asymptotic error bounds which depend on a) an upper bound on relative densities, b) constants associated with local mixing properties of the MCMC dynamics,  and c) the amount of probability mass shifted between effectively disconnected components of the state space as we move from $\mu_0$ to $\mu_n$.\bigskip

The results of this paper fall into two groups: We first consider a simplified model of a multimodal state space and derive error bounds which allow to easily obtain some intuition about the algorithm's ability to cope with multimodality and to study questions of algorithmic design in a relatively non-technical way. These are the results about Sequential MCMC on trees in Section \ref{seqSMCtrees}. In Section \ref{LPlocal} we move on to a more standard Sequential MCMC framework and give similar error bounds which also take into account local changes in the sequence of distributions and local mixing properties.\bigskip

The motivation for studying a Sequential MCMC algorithm on trees stems from the fact that in typical applications the state space splits into more and more effectively disconnected modes over time as we move from $\mu_0$ to $\mu_n$. We project each such disconnected mode to a node in a tree and consider a Sequential MCMC dynamics which moves down the tree from the root, one level in the tree at each step. Thus at each time $k$ the (projected) state space consists of a number of nodes. Each node at level $k$ has at least one successor at level $k+1$. Each successor stands for one disjoint component of the original state space which can only be reached from its predecessor component at level $k$. The role of the ``MCMC'' transitions is limited to allocating particles from the nodes at level $k$ to their successors at level $k+1$. Particles cannot move between nodes at the same level. The latter assumption is in accordance with the fact that transitions between effectively disconnected components of the original unprojected state occur very rarely for the local MCMC dynamics applied in practice. We do not use any mixing properties of the MCMC dynamics since such properties can only be expected to have an effect \textit{within} each disconnected component -- they will not help to correct errors made in allocating particles to modes.\bigskip 

We show that in this reduced setting the algorithm's approximation error can be controlled in terms of a constant which captures how strongly the components gain probability mass over time. Roughly speaking, the algorithm works well if for all $j<k$ no disconnected component under $\mu_j$ carries much less weight than its successors under $\mu_k$. The intuition for this is straightforward: If a node $x$ at time $j$ is much less important under $\mu_j$ than its successors at level $k$, there is a substantial probability that there are no particles in $x$. If $\mu_j(x)$ is small, we may then still have a reasonable particle approximation of $\mu_j$. But if we miss $x$ we also miss its successors at level $k$ and if these are important we obtain a bad approximation of $\mu_k$: Transition states with small weight create a bottleneck for the particle dynamics. These observations follow from a generic non-asymptotic bound for the quadratic error of Sequential MCMC presented in Theorem \ref{thmBound2} and from Proposition \ref{propDJK} which shows how to apply this bound to the tree model.\bigskip

To demonstrate how these results may help to address questions of algorithmic design in a relatively non-technical way, we turn in Section \ref{anexample} to a comparison of resampling particles as opposed to weighting them as is done in the Sequential Importance Sampling method \cite{N01,CMR05}. We provide a detailed analysis of an example where Sequential MCMC with resampling works well while the error of Sequential Importance Sampling increases exponentially over time. This shows that resampling with a finite number of particles can indeed overcome difficulties associated with multimodality in settings where Sequential Importance Sampling fails.\bigskip

Section \ref{LPlocal} contains our second group of results. We consider a more standard Sequential MCMC setting with a sequence of mutually absolutely continuous measures $(\mu_k)_k$ on a common state space $E$.  Instead of the tree structure we consider a sequence of increasingly finer partitions of $E$ and assume that the MCMC dynamics does not move between partition elements. We show that in addition to conditions on the importance gains of disconnected components similar to those in the setting of trees, two additional conditions are sufficient for a good performance of the algorithm: Uniform upper bounds on relative densities between the $\mu_k$ ensure that the importance sampling resampling step works  well. Sufficiently good mixing \textit{within} modes is needed to decorrelate particles after resampling and to explore the state space locally. The mixing conditions we consider follow, e.g., from local hyperboundedness and local Poincaré inequalities for the MCMC steps. These results are developed in three steps: Theorem \ref{thmBound} recalls a generic non-asymptotic error bound for Sequential MCMC from \cite{Schw12}. Proposition \ref{propCon} shows how the constants in the bound of Theorem \ref{thmBound} can be controlled in terms of local $L_{2p}-L_p$-stability conditions for the Feynman-Kac propagator associated with the particle dynamics. Finally, Proposition \ref{corGesamtL} shows how the latter stability conditions follow from the more elementary mixing and boundedness conditions mentioned above.\bigskip 

There is by now a substantial literature on error bounds for Sequential MCMC and related particle systems beginning with the central limit theorems in Del Moral \cite{DM96}, Chopin \cite{C04}, Künsch \cite{K05} and Capp\'e, Moulines and Ryd\'en, \cite{CMR05}. See Del Moral \cite{DM05} for an overview and many results, and Douc and Moulines \cite{DM08} for a recent contribution. Non-asymptotic error bounds, i.e., error bounds for a fixed number of particles are comparatively less studied, see, e.g., Del Moral and Miclo \cite{DM00}, Theorem 7.4.4 of Del Moral \cite{DM05}, C\'erou, Del Moral and Guyader \cite{CDG11} and Whiteley \cite{W11}. The results in the present paper are based on techniques which were developed in Eberle and Marinelli \cite{EM09,EM08} for a related continuous-time particle system and adapted to discrete-time in Schweizer \cite{Schw12}. See these papers for more discussion of the overall approach.\bigskip

The vast majority of the Sequential MCMC literature has focused on the case where the MCMC dynamics mixes well for all $\mu_k$. The only precursors of our results on multimodal target distributions appear to be in Eberle and Marinelli \cite{EM09, EM08} who consider the continuous-time case and restrict attention to the case where MCMC mixes well within the elements of a partition of the state space which is fixed for all $\mu_k$. The case of increasingly finer partitions considered in the present paper is more in line with what is observed in typical applications where $\mu_0$ is easy to sample and $\mu_n$ is multimodal.\bigskip

The idea of reducing complicated multimodal distributions to trees, also known as disconnectivity graphs, has been studied extensively in the chemical physics literature, see Chapter 5 of Wales \cite{W03} for an introduction. Notably, the trees we consider here are only very loosely related to the genealogical trees of the particle system studied in C\'erou, Del Moral and Guyader \cite{CDG11}.\bigskip

A number of related results for the Simulated Tempering and Parallel Tempering algorithms have been proved, in increasing generality in \cite{MZ02, BR04, DSH09, DSH09b}. Basically, tempering algorithms differ from Sequential MCMC by substituting the Importance Sampling Resampling steps with suitable MCMC steps between the levels $\mu_k$. Technically, these results rely on decomposition results for bounding spectral gaps of Markov chains, see \cite{CPS92, MR02, JSTV04}. These decomposition results have the advantage that they do not rely on the assumption that the MCMC dynamics does not move between effectively disconnected components which we made. Therefore, these results can be applied directly to some simple models of interest such as the mean field Ising model. All these results on Tempering algorithms are restricted to simple partition structures with global mixing under $\mu_0$ and good mixing within the components of a fixed partition of the state space for $\mu_1,\ldots, \mu_n$. See, e.g.,  Wales \cite{W03} for many examples from chemical physics which correspond to more general sequences of partitions.\bigskip

The results of the present paper are extracted from a more detailed presentation in the dissertation \cite{Schw11}. Section \ref{Preliminaries} introduces the setting, presents the generic error bound for Sequential MCMC from \cite{Schw12} and proves a variation of the latter error bound. Sections \ref{seqSMCtrees} and
\ref{LPlocal} contain our results on Sequential MCMC for multimodal targets as outlined above. Section \ref{MMrellDiscussion} provides further comparison of the results of Sections \ref{seqSMCtrees} and \ref{LPlocal} and discusses their implications.

\section{Preliminaries}\label{Preliminaries}

Section \ref{notation} introduces the basic notation. Section \ref{tmvmodel} introduces the measure-valued model which is approximated in the algorithm. Section \ref{IPS} introduces the interacting particle system which is simulated in running the algorithm. Section \ref{VariancesWEA} collects some basic results found, e.g., in Del Moral and Miclo \cite{DM00} as well as the generic non-asymptotic error bound from \cite{Schw12} which is applied in the analysis of Section  \ref{LPlocal}. Moreover, we prove a second, related error bound which will be used in Section  \ref{seqSMCtrees}.  The setting introduced here covers the models analyzed in Sections  \ref{seqSMCtrees} and \ref{LPlocal} as special cases.

\subsection{Notation}\label{notation} 

Let $(E,r)$ be a complete, separable metric space and let $\mathcal{B}(E)$ be the $\sigma$-algebra of Borel subsets of $E$. Denote by $M(E)$ the space of finite signed Borel measures on $E$. Let $M_1(E)\subset M(E)$ be the subset of all probability measures. Let $B(E)$ be the space of bounded, measurable, real-valued functions on $E$.\bigskip

For $\mu\in M(E)$ and $f\in B(E)$ define $\mu(f)$ by
\[
\mu(f)=\int_E f(x) \mu(dx)
\]
and $\text{Var}_\mu(f)$ by
\[
\text{Var}_\mu(f)=\mu(f^2)-\mu(f)^2.
\]

Let $(\widetilde{E},\widetilde{r})$ be another Polish space. Consider an integral operator $K(x,A)$ with $K(x,\cdot)\in M(\widetilde{E})$ for $x \in E$ and $K(\cdot, A)\in B(E)$ for $A\in \mathcal{B}(\widetilde{E})$. We define for $\mu\in M(E)$ the measure $\mu K\in M(\widetilde{E})$ by
\[
\mu K(A)=\int_E K(x,A)\mu (dx) \;\;\;\; \forall A \in \mathcal{B}(\widetilde{E}). 
\]
For $f \in B(\widetilde{E})$ we denote by $K(f)\in B(E)$ the function given by
\[
K(f)(x)=K(x,f)= \int_E f(z) K(x,dz)\;\;\;\; \forall x\in E.
\]

\subsection{The Measure-Valued Model}\label{tmvmodel}
Consider a sequence of Polish spaces $(E_k,r_k)$ and a sequence of probability measures $(\mu_k)_{k=0}^n$, $\mu_k \in M_1(E_k)$. This is the sequence of measures we wish to approximate with the algorithm introduced in Section \ref{IPS}. The measures $\mu_k$ are related through
\[
\mu_k(f)= \frac{\mu_{k-1}(g_{k-1,k} K_k(f))}{\mu_{k-1}(g_{k-1,k})} \;\;\;\; \forall f \in B(E_k) 
\]
for positive functions $g_{k-1,k}\in B(E_{k-1})$ and transition kernels $K_k$ with $K_k(x,\cdot)\in M_1(E_{k})$ for $x \in E_{k-1}$ and $K_k(\cdot, A)\in B(E_{k-1})$ for $A\in \mathcal{B}(E_k)$. We define the probability distribution $\hat{\mu}_k \in M_1(E_{k-1})$ by 
\[
\hat{\mu}_k(f)= \frac{\mu_{k-1}(g_{k-1,k} f)}{\mu_{k-1}(g_{k-1,k})} \;\;\;\; \forall f \in B(E_{k-1}).
\]
This implies $\hat{\mu}_k(K_k(f))=\mu_k(f)$ for $f\in B(E_k)$. While we need this slightly more general setting in the analysis of Section \ref{seqSMCtrees}, the example to have in mind is the one where the state spaces $E_k$ are identical and where $K_k$ encompasses many steps of an MCMC dynamics (e.g., Metropolis) with stationary distribution $\mu_k$. In that case $g_{k-1,k}$ becomes an unnormalized relative density between $\mu_{k-1}$ and $\mu_k$ and we have $\hat{\mu}_k=\mu_k$.\bigskip

Next we introduce the Feynman-Kac propagator $q_{j,k}$ which will be the central object of our error analysis. Define the mapping $q_{k-1,k}:B(E_k)\rightarrow B(E_{k-1})$ by
\[
q_{k-1,k}(f)=\frac{g_{k-1,k}K_k(f)}{\mu_{k-1}(g_{k-1,k})}.
\]
Observe that this implies
\[
\mu_{k}(f)=\mu_{k-1}(q_{k-1,k}(f))
\]
Furthermore define for $0 \leq j < k \leq n$ the mapping $q_{j,k}:B(E_k)\rightarrow B(E_j)$ by
\[
q_{j,k}(f)=q_{j,j+1}(q_{j+1,j+2}(\ldots q_{k-1,k}(f)))
\]
and $q_{k,k}(f)=f$.
Note that for $f\in B(E_k)$ we have the relation 
\[
\mu_j(q_{j,k}(f))=\mu_k(f) \;\;\;\text{ for }0\leq j \leq  k \leq n
\]
and the property
\[
q_{j,l} (q_{l,k}(f)) = q_{j,k}(f) \;\;\;\text{ for }0 \leq j <l< k \leq n.
\]

\subsection{The Interacting Particle System}\label{IPS}

In the Sequential MCMC algorithm, we approximate the measures $(\mu_k)_k$ by simulating the interacting particle system introduced in the following. We start with $N$ independent samples $\xi_0=(\xi_0^1,\ldots,\xi_0^N)$ from $\mu_0$. The particle dynamics alternates two steps: Importance Sampling Resampling and Mutation: A vector of particles $\xi_{k-1}$ approximating $\mu_{k-1}$ is transformed into a vector $\hat{\xi}_k$ approximating $\hat{\mu}_k$ by drawing $N$ conditionally independent samples from the empirical distribution of $\xi_{k-1}$ weighted with the functions $g_{k-1,k}$. Afterwards, $\hat{\xi}_k$ is transformed into a vector $\xi_k$ approximating $\mu_k$ by moving the particles $\hat{\xi}_k^i$ independently with the transition kernel $K_k$.\bigskip 

We thus have two arrays of random variables $(\xi_k^j)_{0\leq k\leq n,1\leq j\leq N}$ and $(\hat{\xi}_k^j)_{1\leq k\leq n,1\leq j\leq N}$ where $\xi_k^j$ and $\hat{\xi}_{k+1}^j$ take values in $E_k$. Denote respectively by $\mathbb{P}[\cdot]$ and $\mathbb{E}[\cdot]$ probabilities and expectations taken with respect to the randomness in the particle system, i.e., with respect to the random variables $(\xi_k^j)_{k,j}$ and $(\hat{\xi}_k^j)_{k,j}$. Denote by $\mathcal{F}_k$ the $\sigma$-algebra generated by $\xi_0,\ldots \xi_k$ and $\hat{\xi}_1,\ldots \hat{\xi}_k$ and by $\hat{\mathcal{F}}_k$ the $\sigma$-algebra generated by $\xi_0,\ldots \xi_{k-1}$ and $\hat{\xi}_1,\ldots \hat{\xi}_k$. Denote by $\eta_k^N$ the empirical measure of $\xi_k$, i.e.,
\[
\eta_k^N=\frac1N \sum_{i=1}^N \delta_{\xi_k^i}.
\]
The algorithm proceeds as follows: 

\begin{itemize}
\item[(i)] Draw $\xi_0^1,\ldots,\xi_0^N$ independently from $\mu_0$.
\item[(ii)] For $k=1,\ldots,n$,
	\begin{itemize}
		\item[(a)] draw $\hat{\xi}_k=(\hat{\xi}_k^1,\ldots,\hat{\xi}_k^N)$ according to 
			\[
				\mathbb{P}[\hat{\xi}_k\in dx | \mathcal{F}_{k-1} ]= \prod_{j=1}^N \sum_{i=1}^N \frac{g_{k-1,k}(\xi_{k-1}^i)}{\sum_{l=1}^N 	
					g_{k-1,k}(\xi_{k-1}^l)}\delta_{\xi_{k-1}^i}(dx^j),
			\]
		\item[(b)] draw $\xi_k=(\xi_k^1,\ldots,\xi_k^N)$ according to
			\[
			\mathbb{P}[\xi_{k}\in dx | \hat{\mathcal{F}}_k ]= \prod_{j=1}^N K_k(\hat{\xi}_k^j,dx^j).
			\]
	\end{itemize}
\item[(iii)] Approximate $\mu_n(f)$ by
\[
\eta_n^N(f)=\frac1N \sum_{i=1}^N f(\xi_n^i).
\]
\end{itemize}

In the following we will study, how well $\eta_n^N$ approximates $\mu_n$.

\subsection{Non-asymptotic error bounds}\label{VariancesWEA}
 
We are interested in proving efficient upper bounds for the quantities
\[
\mathbb{E}[|\eta_n^N(f)-\mu_n(f)|^2]
\]
and
\[
\mathbb{E}[|\eta_n^N(f)-\mu_n(f)|].
\]
These quantities can be controlled in terms of the approximation error of a weighted empirical measure $\nu_n^N(f)$ which is easier to handle. We next introduce this measure  $\nu_n^N(f)$ and recall an explicit non-asymptotic upper bound on
\[
\mathbb{E}[|\nu_n^N(f)-\mu_n(f)|^2].
\]
Define for $0\leq k\leq n$
\[
\nu_k^N(f)=\varphi_k \, \eta_k^N(f)
\]
where $\varphi_k$ is given by
\[
\varphi_k=\prod_{j=0}^{k-1} \eta_j^N( q_{j,j+1}(1) ) \text{ for }k \geq 1 \text{ and }\varphi_0=1.
\]
As shown in Del Moral and Miclo \cite{DM00}, $\nu_k^N(f)$ is an unbiased estimator for $\mu_k(f)$, i.e., $E[\nu_k^N(f)]=\mu_k(f)$, and we have  
\begin{equation}\label{erwetak}
\mathbb{E}[\nu_k^N (f)| \mathcal{F}_{k-1}]=\nu_{k-1}^N ( q_{k-1,k}(f) ).
\end{equation}

The connection between the approximation errors of $\eta_n^N(f)$ and $\nu_n^N(f)$ is established in the following lemma.

\begin{lem}\label{EMbounds}
For $f\in B(E_n)$ define $f_n=f-\mu_n(f)$  and denote by $\|\cdot\|_{\text{\textnormal{sup}},n}$ the supremum norm on $B(E_n)$. Then we have the bound
\begin{equation}\label{EMbound1}
\mathbb{E}[(\eta_n^N(f)-\mu_n(f))^2]\leq 2\,\text{\textnormal{Var}}(\nu_n^N(f_n))+2\,\|f_n \|_{\text{\textnormal{sup}},n}^2\text{\textnormal{Var}}(\nu_n^N(1)).
\end{equation}
\end{lem}

See \cite{Schw12} for a proof and a similar bound on the absolute error. Thus we can indeed control the approximation error of $\eta_n^N$ in terms of the approximation error of $\nu_n^N$. We next present the two non-asymptotic error bounds on which the analysis of the later sections relies. These bounds reduce the problem of controlling the particle system to the problem of verifying suitable stability properties of the operators $q_{j,k}$. To this end for $0 \leq j \leq n$,  let $\|\cdot\|_j$ be a norm on the function space $B(E_j)$ such that $\|f\|_j<\infty$ for all $f\in B(E_j)$. Then the central error bound of \cite{Schw12} can be stated as follows:

\begin{thm}\label{thmBound}
For $0 \leq j <k \leq n$, let $c_{j,k}$ be a constant such that for all $f\in B(E_k)$ the following stability inequality for the propagator $q_{j,k}$ is satisfied
\begin{equation}\label{cjninequality}
\max\Big(\|1\|_j\|q_{j,k}(f)^2\|_j,\|q_{j,k}(f)\|_j^2,\|q_{j,k}(f^2)\|_j\Big)  \leq c_{j,k}\|f\|_k^2.
\end{equation}

Define $\widehat{c}_{k}$, $\widehat{v}_k$ and $\varepsilon_k^{N}$ by
\[
\widehat{c}_{k}=\sum_{j=0}^{k-1}c_{j,k} \Big(2+\| q_{j,j+1}(1)-1 \|_j\Big),
\]
by
\[
\widehat{v}_{k}=\sup\left\{\left.
\sum_{j=0}^k \text{Var}_{\mu_j}(q_{j,k}(f)) \right| \,f\in B(E_k),\, \|f\|_k\leq 1
 \right\},
\]
and by
\[
\varepsilon_k^{N}=\sup\Big\{
\mathbb{E}\Big[|\nu_k^N(f)-\mu_k(f)|^2\Big]\, \Big|\,f\in B(E_k),\, \|f\|_k\leq 1
 \Big\}.
\]
Furthermore define
\[
\overline{c}_k=\max_{j\leq k} \widehat{c}_{j},\;\;\;\;\; \overline{v}_k=\max_{j\leq k} \widehat{v}_{j}\;\;\; \text{ and }\;\;\; \overline{\varepsilon}_k^{N}= \max_{j\leq k} \varepsilon_j^N.
\]
Then for all $f\in B(E_n)$ we have
\begin{equation}\label{thmf1c}
N \mathbb{E}\Big[|\nu_n^N(f)-\mu_n(f)|^2\Big] \leq \sum_{j=0}^n \text{\textnormal{Var}}_{\mu_j}(q_{j,n}(f))+\| f \|_n^2 \widehat{c}_{n}\;  \overline{\varepsilon}_n^{N}
\end{equation}
and, if $N\geq 2 \overline{c}_n$,
\begin{equation}\label{thmf2c}
\overline{\varepsilon}_n^{N}\leq 2\frac{\overline{v}_n}{N}.
\end{equation}
\end{thm}

For the analysis of Sequential MCMC on trees we rely on a variation of Theorem \ref{thmBound} which is stated and proved next. The basic difference between the two theorems is that they rely on different expressions for the variance of $\nu_n^N(f)$.

\begin{thm}\label{thmBound2}
For $0 \leq j \leq k \leq n$, let $d_{j,k}$ be a constant such that for all $f\in B(E_k)$ the following stability inequality for the propagator $q_{j,k}$ is satisfied,
\begin{equation}\label{djninequality}
\max\Big(\|1\|_j \|q_{j,k}(f)^2\|_j, \|q_{j,k}(f)\|_j^2 \Big)\leq d_{j,k}\| f\|_k^2.
\end{equation}
Define $\widehat{v}_{k}$, $\overline{v}_k$, $\varepsilon_k^{N}$ and  $\overline{\varepsilon}_k^{N}$ as in Theorem \ref{thmBound} and let
\[
\widehat{d}_k=2 \,\sum_{j=0}^k d_{j,k}
\]
and $\overline{d}_k=\max_{j \leq k} \widehat{d}_k$. Then for all $f\in B(E_n)$ we have
\begin{equation}\label{thmf1c2}
N \mathbb{E}\Big[|\nu_n^N(f)-\mu_n(f)|^2\Big] \leq \sum_{j=0}^n \text{\textnormal{Var}}_{\mu_j}(q_{j,n}(f))+\| f \|_n^2 \widehat{d}_{n}\;  \overline{\varepsilon}_n^{N}
\end{equation}
and, if $N\geq 2 \overline{d}_n$,
\begin{equation}\label{thmf2c2}
\overline{\varepsilon}_n^{N}\leq 2\frac{\overline{v}_n}{N}.
\end{equation}
\end{thm}

The main difference between the crucial conditions (\ref{cjninequality}) and (\ref{djninequality}) in the two bounds is that (\ref{djninequality}) includes the case $j=k$. This implies that we need a constant which allows to bound $\|f^2\|_k$ against $\|f\|_k^2$. This is generally difficult since, unlike in the cases $j<k$, there are no transition kernels $K_l$ on the left hand side whose smoothing properties may be exploited. An important exception is the case where $\|\cdot\|_k$ is a supremum norm since in that case we have $\|f^2\|_k=\|f\|_k^2$. In the setting of Sequential MCMC on trees analyzed in Section \ref{seqSMCtrees} we indeed rely on supremum norms and thus obtain better constants from Theorem \ref{thmBound2}.

\begin{proof}[Proof of Theorem \ref{thmBound2}]
We can write the variance of $\nu_n^N(f)$ as follows:
\begin{equation}\label{VarRep}
\mathbb{E}[|\nu_n^N(f)-\mu_n(f)|^2]=\frac1N \mathbb{E}[\nu_n^N(1)\nu_n^N(f^2)-\mu_n(f)^2] +\frac1N \mathbb{E}\left[\sum_{j=0}^{n-1} U_{j,n}^N(f)\right],
\end{equation}
where 
\begin{equation}\label{UjnN}
U_{j,n}^N(f)=\nu_j^N(1)\nu_j^N(q_{j,n}(f)^2)-\nu_j^N(q_{j,n}(f))^2.
\end{equation}
This result goes back to Del Moral and Miclo \cite{DM00}, for a quick verification combine (10) and (11) in \cite{Schw12}. Now note that by the Cauchy-Schwarz inequality and since $\nu_j^N(\cdot)$ is an unbiased estimator for $\mu_j(\cdot)$, we have for any $g,h\in B(E_j)$
\begin{eqnarray}\label{abschfgc}
|\mathbb{E}[\nu_j^N(g)\nu_j^N(h)-\mu_j(g)\mu_j(h)]|\nonumber \\&\leq&
|\mu_j(g) \mathbb{E}[\nu_j^N(h)-\mu_j(h)]+\mu_j(h)\mathbb{E}[\nu_j^N(g)-\mu_j(g)]|\nonumber \\
&+&\mathbb{E}[|\nu_j^N(g)-\mu_j(g)||\nu_j^N(h)-\mu_j(h)|]
\nonumber \\&\leq&   \| g\|_j \| h\|_j\;\varepsilon_j^{N}.
\end{eqnarray}
Thus adding $\pm \text{\textnormal{Var}}_{\mu_j}(q_{j,n}(f))$ to the definition (\ref{UjnN}) of $U_{j,n}^N(f)$ and applying (\ref{abschfgc}) twice yields
\[
\mathbb{E}[U_{j,n}^N(f)]\leq \text{\textnormal{Var}}_{\mu_j}(q_{j,n}(f)) + S_{j,n}(f) \varepsilon_j^{N},
\]
where 
\[
S_{j,n}(f)=\|1\|_j \|q_{j,k}(f)^2\|_j + \|q_{j,k}(f)\|_j^2. 
\]
Applying (\ref{djninequality}) then yields for $0\leq j <n $ the estimate
\begin{equation}\label{UjnBd}
\mathbb{E}[U_{j,n}^N(f)]\leq \text{\textnormal{Var}}_{\mu_j}(q_{j,n}(f)) + 2\,d_{j,n} \| f\|_n^2 \varepsilon_j^{N}.
\end{equation}
A parallel argument yields
\[
\mathbb{E}[\nu_n^N(1)\nu_n^N(f^2)-\mu_n(f)^2]\leq \text{\textnormal{Var}}_{\mu_n}(f) + \widetilde{S}_{n,n}(f) \varepsilon_j^{N},
\]
where $\widetilde{S}_{n,n}(f)=\|1\|_n \|f^2\|_n$ and thus by (\ref{djninequality})
\begin{equation}\label{UjnBd2}
\mathbb{E}[\nu_n^N(1)\nu_n^N(f^2)-\mu_n(f)^2]\leq \text{\textnormal{Var}}_{\mu_n}(f) + d_{n,n} \| f\|_n^2 \varepsilon_j^{N}.
\end{equation}
With these observations we are prepared to bound the quadratic approximation error of $\nu_n^N(f)$: Bounding (\ref{VarRep}) using (\ref{UjnBd}) and (\ref{UjnBd2}) we obtain
\begin{eqnarray}
N \mathbb{E}[|\nu_n^N(f)-\mu_n(f)|^2]\leq \sum_{j=0}^n\text{Var}_{\mu_j}(q_{j,n}(f)) +2 \, \| f \|_{n}^2 
\sum_{j=0}^{n} d_{j,n}\;\varepsilon_j^{N}\nonumber.
\end{eqnarray}
Bounding $\varepsilon_j^{N}$ by $\overline{\varepsilon}_n^{N}$ and inserting the definition of $\widehat{d}_n$ shows (\ref{thmf1c2}). Optimizing (\ref{thmf1c2}) over $f$ with $\|f\|_n\leq 1$ and over $n$ yields
\[
N \overline{\varepsilon}_n^{N} \leq \overline{v}_n +\overline{d}_n \; \overline{\varepsilon}_n^{N}. 
\]
Choosing $N\geq 2\,\overline{d}_n$ and thus $N-\overline{d}_n \geq \frac{N}{2}$ gives (\ref{thmf2c2}).
\end{proof}

In both, Theorems \ref{thmBound} and \ref{thmBound2}, the coefficient of the leading term in the error bound corresponds to the asymptotic variance  in the central limit theorem for $\nu_n^N$ found in Del Moral and Miclo (\cite{DM00}, p. 45).

\section{Sequential MCMC on Trees}\label{seqSMCtrees}

In this section we study the ability of our Sequential MCMC algorithm to explore a multimodal state space by abstracting from the problem of mixing within modes: We consider the algorithm on a simple tree structure. We assume that our sequence of probability distributions $(\mu_k)_k$ lives on a sequence of state spaces $(I_k)_k$ where the states in $I_{k+1}$ have unique predecessors in $I_k$. Particle movements in the MCMC steps are restricted to moving from a state in $I_k$ to one of its successors in $I_{k+1}$.\bigskip

Section \ref{modeltrees} introduces the model including the notation for the tree structure. Section \ref{smctrees} states the algorithm and the error bounds for this setting. While the algorithm considered here should be viewed as a stylized version of Sequential MCMC which abstracts from problems of local mixing, it nevertheless fits into the framework of Section \ref{Preliminaries}. Section \ref{secSIStrees} introduces an alternative algorithm, Sequential Importance Sampling, which is based on weighting particles instead of resampling them. In Section \ref{anexample} we provide an extensive discussion of an elementary example where the error of Sequential MCMC grows polynomially in the number of levels $n$ while the error of Sequential Importance Sampling increases exponentially fast.

\subsection{The Model}\label{modeltrees}
Consider a sequence of probability distributions $\mu_0, \ldots, \mu_n$ on a sequence of finite state spaces $I_0,\ldots, I_n$. Assume that each $\mu_k$ gives positive mass to each point in its state space $I_k$. Denote by $B(I_{k})$ the bounded measurable functions from $I_k$ to $\mathbb{R}$. We define a tree structure on the sequence of state spaces by introducing for $k\in\{0,\ldots,n-1\}$ the predecessor function $p_k:I_{k+1}\cup\ldots\cup I_n \rightarrow I_k$ which maps $x\in I_l$ to its predecessor in $I_k$ for $l>k$. We assume transitivity of the functions $p_k$, i.e., for $j<k<l$ and $x\in I_l$ we assume that
\[
p_j(p_k(x))= p_j(x).
\]
Denote by $\mathcal{P}(I_k)$ the collection of subsets of $I_k$. Conversely to $p_k$, we define the successor function $s_k:I_{0}\cup\ldots\cup I_{k-1} \rightarrow \mathcal{P}(I_k)$ as follows: For $x \in I_l$ with $0\leq l<k\leq n$ the successors in $I_k$ of $x$ are given by
\[
s_k(x)=\{y \in I_k| p_l(y)=x\}. 
\]
We assume that no branches die out, i.e., for all $x \in I_{0}\cup\ldots\cup I_{n-1}$
\[
s_n(x) \neq \emptyset.
\]
With the additional assumption $|I_0|=1$ we would obtain a genuine tree structure yet this is not needed in the following. For a probability distribution $\mu$ on $I_l$, $l<k$, define the probability distribution $\mu^{\shortrightarrow k}$ on $I_k$ as the projection of $\mu$ to $I_k$: For $x\in I_k$,
\[
\mu^{\shortrightarrow k}(x)=\mu(s_l(x)).
\]

For $0\leq k< n$, denote by $g_{k,\,k+1}\in B(I_{k})$ an unnormalized relative density between $\mu_k$ and $\mu_{k+1}^{\shortrightarrow k}$: For all $f\in B(I_k)$
\[
\mu_{k+1}^{\shortrightarrow k}(f)=\frac{\mu_{k}(f g_{k,\,k+1})}{\mu_{k}(g_{k,\,k+1})}.
\]
Denote by $K_{k+1}:I_k \times I_{k+1} \rightarrow [0,1]$ a Markov transition kernel for which 
\[
\mu_{k+1}(f)=\mu_{k+1}^{\shortrightarrow k}(K_{k+1}(f))
\]
for all $f\in B(I_{k+1})$. Any pair of probability distributions $\mu_k$ and $\mu_{k+1}$ with full support on, respectively, $I_k$ and $I_{k+1}$ can be related through a pair $(g_{k,k+1}, K_{k+1})$. Moreover $K_{k+1}$ is unique and $g_{k,k+1}$ is unique up to a normalizing constant. For $x\in I_k$ and $y\in I_{k+1}$, $K_{k+1}$ is given explicitly by
\[
K_{k+1}(x,y)=\left\{\begin{array}{ll}
\frac{\mu_{k+1}(y)}{\mu_{k+1}(s_{k+1}(x))} & \text{ if } y\in s_{k+1}(x)\\
&\\
0 & \text{ otherwise.}
\end{array}\right.
\]

The tree structure, concretely, the fact that the states in $I_k$ are not connected by $K_k$, is a simple model of a multimodal state space: The elements of $I_k$ stand for components of a continuous state space which are separated by regions of very low probability. For the particle dynamics we consider subsequently, the consequence is that particles can move between different branches only through the resampling step but not through the mutation step. This is consistent with our aim of studying, how helpful the resampling step is in overcoming problems associated with multimodality.\bigskip

This model is a special case of the framework of Section \ref{tmvmodel}. The following lemma gives an explicit expression for $q_{j,k}(f)$.

\begin{lem}\label{qjkTrees}
For $0\leq j<k \leq n$, $f\in B(I_k)$ and $x\in I_j$ we have
\begin{equation}\label{qjkTrees1}
q_{j,k}(f)(x) =\frac{\mu_k\Big(f\, 1_{\{s_k(x)\}}\Big)}{\mu_j(x)}.
\end{equation}
In particular,
\begin{eqnarray}\label{qjkTrees2}
|q_{j,k}(f)(x)| &\leq& \left(\max_{y\in I_k} |f(y)|\right)  q_{j,k}(1)(x)\\
&\leq& \left(\max_{y\in I_k} |f(y)|\right) \left(\max_{z\in I_j} \frac{\mu_k^{\shortrightarrow j}(z)}{\mu_j(z)}\right).\nonumber\\\nonumber
\end{eqnarray}
\end{lem}

\begin{proof}[Proof of Lemma \ref{qjkTrees}]
Observe that for $x\in I_j$ and $f \in B(I_k)$ we have for $x \neq y \in I_j$
\[
q_{j,k}(f\,1_{\{ s_k(y) \}})(x)=0.
\]
Thus we can write
\[
q_{j,k}(f)(x)=\sum_{y\in I_j}q_{j,k}\Big(f1_{\{ s_k(y) \}}\Big)(x) =q_{j,k}\Big(f\,1_{\{ s_k(x) \}}\Big)(x),
\]
since $q_{j,k}(f)$ is linear in $f$. Therefore we have
\begin{eqnarray}
 \mu_k\Big(f 1_{\{s_k(x)\}}\Big)=\mu_j\Big(q_{j,k}\Big(f 1_{\{s_k(x)\}}\Big)\Big)= \mu_j(x) q_{j,k}\Big(f 1_{\{s_k(x)\}}\Big)(x)=\mu_j(x)q_{j,k}(f)(x),\nonumber
\end{eqnarray}
which can be rearranged into (\ref{qjkTrees1}). (\ref{qjkTrees2}) follows from 
\[
|q_{j,k}\Big(f\,1_{\{ s_k(x) \}}\Big)(x)| \leq \left(\max_{y\in I_k} |f(y)|\right) q_{j,k}\Big(1_{\{ s_k(x) \}}\Big)(x)=\left(\max_{y\in I_k} |f(y)|\right) \frac{\mu_k(s_k(x))}{\mu_j(x)}
\]
and the definition of $\mu_k^{\shortrightarrow j}$.
\end{proof}

\subsection{Sequential MCMC}\label{smctrees}\label{errorbdsTrees}

We next apply to our model the error bounds of Theorem \ref{thmBound2}. To achieve this we need to define a series of norms $\|\cdot\|_j$ on $B(I_j)$ and find constants $d_{j,k}$ such that the inequality 
\begin{equation}\label{djkinequality}
\max\Big(\|1\|_j\|q_{j,k}(f)^2\|_j,\|q_{j,k}(f)\|_j^2\Big)  \leq d_{j,k}\;\|f\|_k^2
\end{equation}
is satisfied. We choose $\|\cdot\|_j$ to be the maximum-norm on $B(I_j)$, i.e., for $f\in B(I_j)$
\[
\|f\|_j=\max_{x\in I_j} |f(x)|.
\]
The following proposition gives a choice of constants $d_{j,k}$ which guarantee that (\ref{djkinequality}) is satisfied and shows that these constants can also be used to bound the remaining quantities arising in the error bound of Theorem \ref{thmBound2}:

\begin{prop}\label{propDJK}
For all $j<k$, inequality (\ref{djkinequality}) is satified for the constants
\[
d_{j,k}=\left(
\max_{x\in I_j} \frac{\mu_k^{\shortrightarrow j}(x)}{\mu_j(x)} 
\right)^2.
\]
Moreover for all $f\in B(E_k)$
\begin{equation}\label{varmungrob}
\text{\textnormal{Var}}_{\mu_j}(q_{j,k}(f))\leq \sqrt{d_{j,k}}\,\|f\|_k^2,
\end{equation}
as well as
\[
\widehat{v}_{k}\leq \sum_{j=0}^k \sqrt{d_{j,k}} \text{ and } \overline{v}_{k}\leq \max_{i \leq k}\sum_{j=0}^i \sqrt{d_{j,i}}.
\]
\end{prop}
\begin{proof}[Proof of Proposition \ref{propDJK}]
Observe that we have $\|f^2\|_j=\|f\|_j^2$, $\|1\|_j=1$ and by Lemma \ref{qjkTrees}
\[
\|q_{j,k}(f)\|_j \leq \|q_{j,k}(1)\|_j \|f\|_n.
\]
Moreover by the same lemma we have 
\begin{equation}\label{greins}
\|q_{j,k}(1)\|_j=\max_{x\in I_j} \frac{\mu_k^{\shortrightarrow j}(x)}{\mu_j(x)}.
\end{equation}
Thus (\ref{djkinequality}) is satisfied for the constants $d_{j,k}$ we defined.

For the bound on $\text{\textnormal{Var}}_{\mu_j}(q_{j,k}(f))$ note that
\[
\text{\textnormal{Var}}_{\mu_j}(q_{j,k}(f))\leq \mu_j(q_{j,k}(f)^2)\leq \|f\|_k^2\; \|q_{j,k}(1)\|_k\; \mu_j(q_{j,k}(1))\leq \sqrt{d_{j,k}}\,\|f\|_k^2.
\]
This immediately implies the bounds on $\widehat{v}_{k}$ and $\overline{v}_{k}$.
\end{proof}

In order to apply the error bound of Theorem \ref{thmBound2} it is sufficient to control the constants $d_{j,k}$ defined in the proposition. $d_{j,k}$ is large when a node at level $j$, which carries little mass, has offspring at level $k$, which (in sum) carries considerably more probability mass. Notably, the constant $d_{j,k}$ does not take into account any further branching of the state space which occurs at levels $j+1,\ldots,n$. We next set these bounds into perspective by deriving a lower bound on the asymptotic variance 
\[
\text{\textnormal{Var}}^{\text{as}}_k (f) =\sum_{j=0}^k \text{\textnormal{Var}}_{\mu_j}(q_{j,k}(f))
\]
for the test function $f\equiv 1 \in B(I_k)$.

\begin{prop}\label{assvartest}
\[
\text{\textnormal{Var}}^{\text{as}}_k (1) = \sum_{j=0}^k \sum_{x\in I_j} \mu_k^{\shortrightarrow j} (x) \left(\frac{\mu_k^{\shortrightarrow j} (x)}{\mu_j(x)}  -1\right)= \sum_{j=0}^k\mu_k^{\shortrightarrow j} (q_{j,k}(1)-1).
\]
\end{prop}
\begin{proof}[Proof of Proposition \ref{assvartest}]
Observe that
\[
\text{\textnormal{Var}}_{\mu_j}(q_{j,k}(1))= \left[
\sum_{x\in I_j} \mu_{j}(x)\left(q_{j,k}(1)(x)\right)^2
\right]-\mu_j(q_{j,k}(1))^2.
\]
By Lemma \ref{qjkTrees} we have
\[
\left(q_{j,k}(1)(x)\right)^2=\left(\frac{\mu_k^{\shortrightarrow j}(x)}{\mu_j(x)}\right)^2.
\]
By the fact that
\[
\mu_j(q_{j,k}(1))^2=\mu_k(1)^2=1=\sum_{x\in I_j} \mu_k^{\shortrightarrow j} (x),
\]
we can thus write
\[
\text{\textnormal{Var}}_{\mu_j}(q_{j,k}(1))=
 \sum_{x\in I_j} \mu_k^{\shortrightarrow j}(x) \left( \frac{\mu_k^{\shortrightarrow j}(x)}{\mu_j(x)}-1\right)=\mu_k^{\shortrightarrow j}(q_{j,k}(1)-1).
\]
Summing over $j$ completes the proof.
\end{proof}

Denote the expression for $\text{\textnormal{Var}}_{\mu_j}(q_{j,k}(1))$ from the proposition by $v_{j,k}$, i.e.,
\[
v_{j,k}= \sum_{x\in I_j} \mu_k^{\shortrightarrow j}(x) \left( \frac{\mu_k^{\shortrightarrow j}(x)}{\mu_j(x)}-1\right). 
\]
  
$d_{j,k}$ may be large even when $v_{j,k}$ is small: $d_{j,k}$ is large if the successors at level $k$ of $x\in I_j$ are  -- relatively -- much more important under $\mu_k$ than $x$ is under $\mu_j$. In this case $v_{j,k}$ may still be small if the absolute importance of the successors of $x$ is small under $\mu_k$. In short, $v_{j,k}$ may be much smaller than $d_{j,k}$ if the largest (relative) gains in importance are made by regions of the state space that remain (absolutely) unimportant.\bigskip

As a by-product, note that from the proof of Proposition \ref{assvartest} we immediately get an upper bound on $\text{\textnormal{Var}}_{\mu_j}(q_{j,k}(f))$ which is sharper than (\ref{varmungrob}): 
\[
\text{\textnormal{Var}}_{\mu_j}(q_{j,k}(f))\leq  \mu_j(q_{j,k}(1)^2) \, \|f\|_k^2 =\widetilde{d}_{j,k} \|f\|_k^2, 
\]
where $\widetilde{d}_{j,k}$ is defined as
\[
\widetilde{d}_{j,k}
=
\sum_{x\in I_j}\frac{\mu_k^{\shortrightarrow j}(x)^2}{\mu_j(x)}
=\mu_k^{\shortrightarrow j} (q_{j,k}(1)).
\]
We obtain corresponding sharper upper bounds on $\widehat{v}_k$ and $\overline{v}_k$. This allows to bound the leading term in the error bounds of Theorem \ref{thmBound2} using $\widetilde{d}_{j,k}$ in place of $\sqrt{d_{j,k}}$.

\subsection{Sequential Importance Sampling}\label{secSIStrees} 

For the purpose of comparison, we next introduce a Sequential Importance Sampling algorithm for the tree model and give an explicit expression of its approximation error for a class of test functions.\bigskip

In Sequential Importance Sampling, particles are moved independently according to the kernels $K_k$. Afterwards, importance weights $\omega$ are calculated for the particles which allow to obtain an estimator for $\mu_n$ through a weighted empirical measure of the particles. In the present framework, Sequential Importance Sampling is equivalent to simple Importance Sampling between the probability distribution $\pi_n$ on $I_n$ given by
\[
\pi_n=\mu_0 K_{1}\ldots K_{n}
\]
and $\mu_n$. For simplicity, we consider only unnormalized Importance Sampling, i.e., we assume that we can calculate the weights exactly (and not only up to a normalizing constant). This has the advantage that we do not have to consider a bias introduced by normalizing the particle weights through their sum. Otherwise, the algorithm corresponds to, e.g., the Annealed Importance Sampling algorithm of Neal \cite{N01}.\bigskip

Instead of a system of particles, it is thus sufficient to consider only the vector of particles $(\widetilde{\xi}_n^i)_{1\leq i \leq N}$ which are distributed independently according to $\pi_n$. We define the importance weight function $\omega_n\in B(I_n)$ by
\[
\omega_n(x)=\frac{\mu_n(x)}{\pi_n(x)}, \text{ for all } x\in I_n 
\]
Then for $f\in B(I_n)$ the Sequential Importance Sampling estimator $\widetilde{\eta}_n(f)$ is given by
\[
\widetilde{\eta}_n(f)=\frac{1}{N}\sum_{i=1}^N f\left(\widetilde{\xi}_n^i\right) \omega_n\left(\widetilde{\xi}_n^i\right).
\]
$\widetilde{\eta}_n(f)$ is an unbiased estimator for $\mu_n(f)$, i.e.,
\[
\mathbb{E}[\widetilde{\eta}_n(f)]=\mu_n(f).
\]

We next calculate a formula for the quadratic approximation error for test functions of the form $f=1_{\{x\}}$.

\begin{lem}\label{SISerror}
For $x\in I_n$ and $f=1_{\{x\}}$ we have
\[
\mathbb{E}[|\widetilde{\eta}_n(f)-\mu_n(f)|^2]=\frac{\mu_n(x)^2}{N}\left(\frac{1}{\pi_n(x)}-1  \right).
\]
\end{lem}

\begin{proof}[Proof of Lemma \ref{SISerror}]
To prove the lemma we only need the following calculation based on the unbiasedness of $\widetilde{\eta}_n(f)$:
\begin{eqnarray}
\mathbb{E}[|\widetilde{\eta}_n(f)-\mu_n(f)|^2]&=&\mathbb{E}\left[\left(\left(
\frac1N \sum_{i=1}^N\frac{\mu_n(x)}{\pi_n(x)}1_{\{x \}}\left(\widetilde{\xi}_n^i\right)
\right)
-\mu_n(x)
\right)^2
\right]\nonumber\\
&=& \frac{1}{N^2}\sum_{i=1}^N \mathbb{E}\left[\left(
\frac{\mu_n(x)}{\pi_n(x)}1_{\{x \}}\left(\widetilde{\xi}_n^i\right) - \mu_n(x)
\right)^2
\right]\nonumber\\
&=& \frac1N\left(
\pi_n(x)\left(\frac{\mu_n(x)}{\pi_n(x)} -\mu_n(x) \right)^2+(1-\pi_n(x))\mu_n(x)^2
\right)\nonumber\\
&=&\frac{\mu_n(x)^2}{N}\left(\frac{1}{\pi_n(x)}-1\right).
\nonumber
\end{eqnarray}
\end{proof}

We thus see that Sequential Importance Sampling can only perform well if the distribution $\pi_n$ is sufficiently close to $\mu_n$, more precisely, if no state which is unimportant under $\pi_n$ is important under $\mu_n$.

\subsection{Example: Weighting or Resampling?}\label{anexample}

We now apply the error bounds we just developed to a concrete example depicted in Figure \ref{bildbaum}. Our aim is to show that in this case Sequential MCMC, notably, its Resampling step, succeeds in a multimodal setting in which Sequential Importance Sampling severely suffers from weight degeneracy. Section \ref{treemodel} introduces the setting of the example. Section \ref{treesemigroup} derives upper bounds on $q_{j,k}(1)$. Sections \ref{errSMCMCtreeEx} and \ref{wdSIS} contain the error analysis for, respectively, Sequential MCMC and Sequential Importance Sampling. Section \ref{moreexamples} closes our comparison of Sequential MCMC and Sequential Importance Sampling by discussing some further examples.

\subsubsection{The Model}\label{treemodel}

We consider the sequence of state spaces $I_0,\ldots, I_n$ given by
\[
I_k=\{0_k,\ldots, k_k \}.
\]
Thus the elements of $I_k$ are the natural numbers from $0$ to $k$, indexed by $k$ in order to keep the notation clearer. 
\begin{figure}
\begin{center}
\includegraphics[height=250pt]{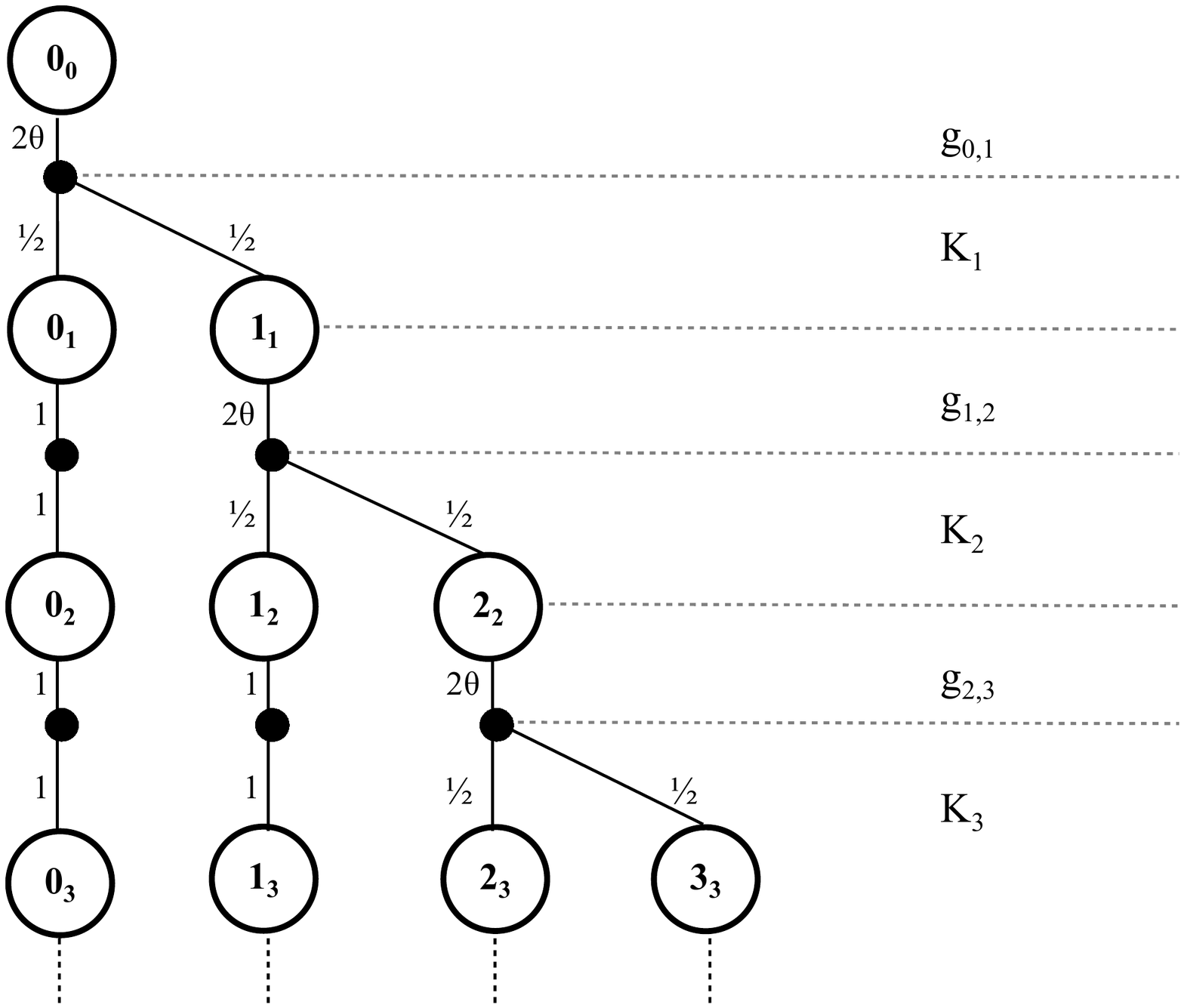}
\end{center}
\caption{Weighting or Resampling?}
\label{bildbaum}
\end{figure}
For $l>k$, the predecessor in $I_k$ of $j_l\in I_l$ is given by $j_k$ if $j\leq k$, otherwise it is $k_k$:
\[
p_k(j_l)=\left\{
\begin{array}{ll}
j_k,& \text{ if } j \leq k \\
k_k, & \text{ if } j> k. \\
\end{array}
\right.
\]
We thus have a simple tree structure where from level $k$ to level $k+1$ the ``largest'' node $k_k$ has two successors, $k_{k+1}$ and $(k+1)_{k+1}$, while all other nodes $j_k$ have only one successor $j_{k+1}$. Accordingly, for $l>k$ and $j_k \in I_k$, the successor function is given by
\[
s_l(j_k)=\left\{
\begin{array}{ll}
\{j_l\},& \text{ if } j< k\\
\{k_l, \ldots, l_l\}, & \text{ if } j= k.\\
\end{array}
\right.
\]

We define the sequence $\mu_0,\ldots, \mu_n$ implicitly through $g_{k,k+1}$ and $K_{k+1}$. We choose the weight function $g_{k,k+1}\in B(I_k)$ such that only the mass of $k_k$ is modified while the relative masses of the other nodes remain the same:
\[
g_{k,k+1}(j_k)=\left\{
\begin{array}{ll}
1,& \text{ if } j< k\\
2 \theta & \text{with $\theta >0$, if } j= k.\\
\end{array}
\right.
\]
The transition kernel $K_{k+1}:I_k\times I_{k+1}\rightarrow [0,1]$  is chosen such that $K_{k+1}(j_k,\cdot)$ is the uniform distribution on the successors of $j_k$:
\[
K_{k+1}(j_k,i_{k+1})=\left\{
\begin{array}{ll}
1& \text{ if } i=j< k\\
\frac{1}{2}  & \text{ if } j= k \text{ and }i\in\{k,k+1\}\\
0  & \text{ otherwise.}\\
\end{array}
\right.
\]
Observe that for $\theta >\frac{1}{2}$ we have two countervailing effects: One from the kernels $K_k$ and one from the functions $g_{k,k+1}$. On the one hand, the kernels $K_k$ favor that mass is concentrated on $j_k$ with small $j$. If we had a constant function $g_{k,k+1}$ (i.e. $\theta=\frac12$), $\mu_k$ would be a geometric distribution with parameter $\frac12$ and maximum in $0_k$ . On the other hand, the weight functions $g_{k,k+1}$ move mass to the largest node $k_k$. As becomes clear from the explicit formula for $\mu_k$ calculated next, the case of  $\theta > 1$ which we mainly consider is the case where the second effect is sufficiently strong in the sense that $\mu_k(k_k)>\mu_k(j_k)$ for $j<k-1$. As $\theta$ approaches $1$, $\mu_k$ converges to the uniform distribution on $I_k$. The cases where $\theta < 1$ are largely omitted in our error bounds, not because they are more difficult, but because they are less interesting and would need a largely separate analysis.

\begin{cor}\label{muktrees}
For $j_k\in I_k$ we have
\[
\mu_k(j_k)=\left\{
\begin{array}{ll}
\frac{\theta^{j+1}}{Z_k},& \text{ if } j< k\\&\\
\frac{\theta^{k}}{Z_k},& \text{ if } j= k,\\
\end{array}
\right.
\]
where the normalizing constant $Z_k$ is given by
\begin{eqnarray}
Z_k&=& \theta^k +\sum_{j=0}^{k-1}\theta^{j+1}\label{Zk1}.
\end{eqnarray}
Moreover for $\theta \neq 1$,
\begin{eqnarray}
Z_k&=& \theta^k+\frac{\theta}{\theta-1}(\theta^k-1).\label{Zk2}
\end{eqnarray}
\end{cor}

The corollary is an immediate consequence of our choices of $g_{k,k+1}$ and $K_{k+1}$. Thus for $\theta > 1$,  $\mu_k$ can be characterized as follows: It is a geometric distribution with maximum in $(k-1)_k$ on $0_k,\ldots, (k-1)_k$. Additionally we have $\mu_k((k-1)_k)=\mu_k(k_k)$.

\subsubsection{Controlling $q_{j,k}$}\label{treesemigroup}

From here on we mostly focus on the case $\theta \geq 1$. In order to apply the error bounds of Section \ref{errorbdsTrees}  we have to study the expressions $q_{j,k}(1)$ for this example. This is begun in the following lemma.

\begin{lem}\label{qjktreeex}
For $0\leq k <l \leq n$, we have
\[
q_{k,l}(1)(j_k)=\left\{
\begin{array}{ll}
\frac{Z_k}{Z_l}& \text{ if } j< k\\&\\
\frac{Z_k Z_{l-k}}{Z_l}& \text{ if } j= k.\\
\end{array}
\right.
\]
Furthermore for $\theta\geq 1$,
\[
\max_{j_k\in I_k} q_{k,l}(1)(j_k) = \frac{Z_k Z_{l-k}}{Z_l}.
\]
\end{lem}

\begin{proof}[Proof of Lemma \ref{qjktreeex}]
Recall from Lemma \ref{qjkTrees} that
\[
q_{k,l}(1)(j_k)=\frac{\mu_l(s_l(j_k))}{\mu_k(j_k)}.
\]
Thus for $j_k \neq k_k$ Corollary \ref{muktrees} immediately implies
\[
q_{k,l}(1)(j_k)=\frac{\mu_l(j_l)}{\mu_k(j_k)}=\frac{Z_k}{Z_l}.
\]
For $j_k = k_k$ we have
\begin{eqnarray}
q_{k,l}(1)(k_k)&=&\frac{\mu_l(\{k_l,\ldots,l_l\})}{\mu_k(k_k)}\nonumber\\
&=& \frac{Z_k}{Z_l} \left(
\frac{\theta^l + \sum_{i=k}^{l-1}\theta^{i+1}
}{\theta^k}
\right)\nonumber\\
&=&\frac{Z_k}{Z_l} \left(
\theta^{l-k}+\sum_{i=0}^{l-k-1}\theta^{i+1}
\right)\nonumber\\
&=&
\frac{Z_k Z_{l-k}}{Z_l}.\nonumber
\end{eqnarray}

Observe from (\ref{Zk1}) that $Z_k<Z_l$ and thus for $j_k\neq k_k$ $q_{k,l}(1)(j_k)<1$. Since both $\mu_k$ and $\mu_l$ are probability measures and since $\mu_k(q_{k,l}(1))=\mu_l(1)=1$ this implies
\[
\max_{j_k \in I_k} q_{k,l}(1)(j_k)=q_{k,l}(1)(k_k)>1.
\]
\end{proof}

Thus in order to control $q_{k,l}(1)$ we need bounds on the constants $Z_k$. The following lemma gives two pairs of bounds on $Z_k$. The bounds in (\ref{Zkbd1}) get sharp as $\theta$ approaches $1$ while the bounds in (\ref{Zkbd2}) get sharp as $\theta$ gets large.

\begin{lem}\label{Zkbds}
We have for $\theta\geq 1$
\begin{equation}\label{Zkbd1}
(k+1) \theta \leq Z_k \leq (k+1)\theta ^k
\end{equation}
and 
\begin{equation}\label{Zkbd2}
2 \theta^k \leq Z_k\;\;\;\;\; \;\;\;\;\text{   and, if    } \theta >1,\;\;\;\; \text{    } Z_k \leq  \rho(\theta) \theta^k
\end{equation}
where we define
\begin{equation}\label{defrho}
\rho(\theta)=2+\frac{1}{\theta-1}.
\end{equation}
\end{lem}

\begin{proof}[Proof of Lemma \ref{Zkbds}]
The bounds in (\ref{Zkbd1}) and the lower bound in (\ref{Zkbd2}) follow immediately from (\ref{Zk1}) and from the fact that for $k>i$ we have $\theta^k \geq \theta^i$. The upper bound in (\ref{Zkbd2}) follows from (\ref{Zk2}) since
\[
Z_k=\theta^k+\frac{\theta}{\theta-1}(\theta^k-1)< \left(1+\frac{\theta}{\theta-1}\right)\theta^k=
\left(2+\frac{1}{\theta-1}\right)\theta^k
.
\]
\end{proof}

 We thus arrive at the following upper bound on $\|q_{k,l}(1)\|_{k}$.

\begin{cor}\label{qklbds}
For $k<l$ and $\theta >1$ we have 
\[
\|q_{k,l}(1)\|_k \leq \min\left(
\frac{\rho(\theta)^2}{2},
\frac{\rho(\theta)^2}{l+1}\,\theta^{l-1},
\frac{(l+2)^2}8,
\frac{l+2}{2} \,\theta^{l-1}
\right).
\]
\end{cor}
\begin{proof}[Proof of Corollary \ref{qklbds}]
By combining each time one lower bound and one upper bound from Lemma \ref{Zkbds} we obtain four upper bounds on
\[
\|q_{k,l}(1)\|_k =\frac{Z_{k}Z_{l-k}}{Z_l}.
\]
Applying the inequalities $(k+1)(l-k+1)\leq \frac14 (l+2)^2$ and $(k+1)(l-k+1)/(l+1)\leq \frac{1}{2}(l+2)$ completes the proof.
\end{proof}

For  $\theta$ sufficiently close to $1$, the upper bound 
\begin{equation}\label{thetabdl}
\|q_{k,l}(1)\|_k\leq 
\frac{l+2}{2} \theta^{l-1},
\end{equation}
which is obtained from using both directions of (\ref{Zkbd1}), is the sharpest one. For sufficiently large $\theta$, the bound
\begin{equation}\label{thetabd}
\|q_{k,l}(1)\|_k\leq 
\frac{\rho(\theta)^2}{2} 
\end{equation}
obtained from (\ref{Zkbd2}) is best. Depending on the values of $k$ and $l$, one of the two other bounds may be even better for intermediate values of $\theta$. Finally, note that the third and fourth bounds also apply to $\theta=1$ since they do not rely on the upper bound from (\ref{Zkbd2}).\bigskip

It is quite intuitive, that for $\theta \approx 1$ our bounds on $q_{j,k}(1)$ depend more sensitively on $k$. With a large value of $\theta$ mass is concentrated quickly in the highest branch of the tree such that the sequence $a_k=\mu_{k}(k_k)$ varies relatively little in $k$. For $\theta \approx 1$, mass is accumulated only slowly in $k_k$  as $k$ increases such that the same sequence $a_k$ is increasing substantially in $k$ at least for small values of $k$. This increase is reflected in the fact that our upper bound on $q_{j,k}(1)$ is increasing with $k$ in that case. Put differently, for $\theta$ close to $1$ and $k$ not large, the distributions $\mu_k$ are not very concentrated (i.e. close to the uniform distribution) and thus more costly to approximate. As we will see, the approximation error of our algorithm is indeed of  worse order in $n$ at $\theta=1$ than for $\theta>1$ (or $\theta<1$). This can also be seen as an elementary manifestation of the critical slowing down phenomenon. 

\subsubsection{Error Bounds for Sequential MCMC}\label{errSMCMCtreeEx}

In the following we give two error bounds, both based on Theorem \ref{thmBound2}: one which degenerates as $\theta$ approaches $1$, and one which does not degenerate but which is worse for $\theta$ sufficiently greater than $1$. Before we begin, note that a dependence on the parameter $n$ enters the error bound from two sources. While the two terms of the error bound of Theorem \ref{thmBound2} are, respectively, linear and quadratic in $n$, we obtain a stronger dependence on $n$ in Proposition \ref{errtreeextheta1} below since $n$ is also the size of the state space $I_n$ and a parameter of the distribution $\mu_n$. To confirm that this difference between the results is not an artefact of our upper bounds, we calculate the asymptotic variance in the case $\theta=1$ explicitly in Lemma \ref{lemAssVarEx} at the end of this section.\bigskip

The first result, for $\theta$ sufficiently greater than $1$, is based on the bound (\ref{thetabd}), i.e., we choose
\[
\|q_{k,l}(1)\|_k^2\leq d_{k,l}=\frac{\rho(\theta)^4}{4},
\]
with $\rho(\theta)$ as defined in (\ref{defrho}).

\begin{prop}\label{errtreeextheta2}
Consider $\theta > 1$,  $N> \rho(\theta)^4\,(n+1)$ and $f\in B(I_n)$. Then we have
\[
\mathbb{E}[|\nu_n^N(f)-\mu_n(f)|^2] \leq \|f\|_n^2\,\left( \frac{\rho(\theta)^2}{2}\frac{n+1}{N} + \rho(\theta)^6\, \frac{(n+1)^2}{N^2} \right).
\] 
\end{prop}

\begin{proof}[Proof of Proposition \ref{errtreeextheta2}]
In order to apply Theorem \ref{thmBound2} we have to control the constants introduced therein. By our choice of $d_{j,k}$, we get
\[
\widehat{d}_k \leq \frac{\rho(\theta)^4 (k+1)}{2}.
\]
Since this upper bound is increasing in $k$ it also applies to $\overline{d}_k$. Furthermore by (\ref{varmungrob}) we have
\[
\sum_{j=0}^n \text{\textnormal{Var}}_{\mu_j}(q_{j,n}(f))\leq \frac{\rho(\theta)^2}{2}(n+1) \,\|f\|_n^2.
\]
This implies
\[
\widehat{v}_{k} \leq  \frac{\rho(\theta)^2}{2}(k+1)
\]
and since this upper bound is increasing in $k$ it also applies to  $\overline{v}_{k}$. Inserting these results into Theorem \ref{thmBound2} completes the proof.
\end{proof}

These bounds degenerate quickly as $\theta$ approaches $1$ since $\rho(\theta)$ gets arbitrarily large then. To demonstrate that we obtain reasonable constants in our bounds for sufficiently large $\theta$, we give the following result derived from the special case $\theta=2$ and thus $\rho(2)=3$.

\begin{cor}
Consider $\theta \geq 2$,  $N> 81 (n+1)$ and $f\in B(I_n)$. Then we have
\[
\mathbb{E}[|\nu_n^N(f)-\mu_n(f)|^2] \leq \|f\|_n^2\,\left( \frac{9}{2}\,\frac{n+1}{N} + 729\,\frac{(n+1)^2}{N^2} \right).
\] 
\end{cor}

We now turn to a bound which does not degenerate at $\theta=1$. For the sake of simplicity we rely on the bound
\begin{equation}
\|q_{k,l}(1)\|_k^2\leq d_{k,l}= 
 \frac{(l+2)^4}{64}
 \end{equation}
from Corollary \ref{qklbds} instead of the bound (\ref{thetabdl}) which is sharper  and has a better order in $l$ for $\theta \approx 1$ but which degenerates quickly as $\theta$ increases. 

\begin{prop}\label{errtreeextheta1}
Consider $\theta \geq 1$,  $N> \frac{1}{16} (n+2)^5$ and $f\in B(I_n)$. Then we have
\[
\mathbb{E}[|\nu_n^N(f)-\mu_n(f)|^2] \leq \|f\|_n^2\,\left(\frac18\, \frac{ (n+2)^3}{N}
 + \frac{1}{128} \,\frac{ (n+2)^8}{  N^2} \right).
\] 
\end{prop}

\begin{proof}[Proof of Proposition \ref{errtreeextheta1}]
By our choice of $d_{j,k}$, we get
\[
\widehat{d}_k \leq \frac{(k+2)^5}{32}.
\]
Since this bound is increasing in $k$ it also applies to $\overline{d}_k$. Furthermore by (\ref{varmungrob}), we have
\[
\sum_{j=0}^n \text{\textnormal{Var}}_{\mu_j}(q_{j,n}(f))\leq \frac{(n+2)^3}{8} \,\|f\|_n^2, \text{ and } \widehat{v}_{k}\leq  \frac{(k+2)^3}{8}.
\]
Since the latter bound is increasing in $k$ it also applies to $\overline{v}_k$. Inserting these results into Theorem \ref{thmBound2} completes the proof.
\end{proof}

As noted above we used in Proposition \ref{errtreeextheta1} a bound of order $n^4$ on $\|q_{k,n}(1)\|_k^2$ instead of relying on (\ref{thetabdl}) which may have led to a better order at least for $\theta$ close to $1$. Thus we expect that the error bound of Proposition \ref{errtreeextheta1} can be improved concerning the order in $n$. In Section \ref{wdSIS}, we show however that the approximation error of Sequential Importance Sampling is growing exponentially in $n$ in this example. Thus Proposition \ref{errtreeextheta1} is strong enough to make our point that the resampling step in our Sequential MCMC algorithm overcomes the problem of weight degeneracy.\bigskip

To close our analysis of the error bound for $\theta$ close to $1$, we explicitly calculate the asymptotic variance -- and thus the leading coefficient in the error bound of Theorem \ref{thmBound2} -- for the case $\theta=1$ and $f \equiv 1\in B(I_n)$. This asymptotic variance is quadratic in $n$ which proves that it is no artifact of our upper bounds, that we do not achieve as good an order in $n$ in Proposition \ref{errtreeextheta1} as in Proposition \ref{errtreeextheta2}.

\begin{lem}\label{lemAssVarEx}
For $\theta=1$ we have 
\[
\text{\textnormal{Var}}^{\text{as}}_n (1) =\sum_{j=0}^n \text{\textnormal{Var}}_{\mu_j}(q_{j,n}(1))=\frac{n^2(n-1)}{12 (n+1)}.
\]
\end{lem}

\begin{proof}[Proof of Lemma \ref{lemAssVarEx}]
By Proposition \ref{assvartest} we have
\[
\text{\textnormal{Var}}^{\text{as}}_n (1) =\sum_{j=0}^n w_j,
\;\text{ where}\;
w_j=\sum_{x\in I_j} \mu_n^{\shortrightarrow j}(x)\left(
\frac{\mu_n^{\shortrightarrow j}(x)}{\mu_j(x)}-1\right).
\]
Now observe that for $\theta=1$ and for all $x \in I_j$ we have
\[
\mu_j(x)=\frac{1}{j+1}
\]
and
\[
\mu_n^{\shortrightarrow j}(x)=\left\{\begin{array}{ll}
\frac{n-j+1}{n+1} & \text{for } x=j_j,\\&\\
\frac{1}{n+1} & \text{otherwise}.
\end{array}\right.
\]
Thus we have
\begin{eqnarray}
w_j&=& \sum_{x\in I_j} \mu_n^{\shortrightarrow j} (x)((j+1)\mu_n^{\shortrightarrow j} (x)-1)\nonumber\\
&=&-1+(j+1) \sum_{x\in I_j} \mu_n^{\shortrightarrow j} (x)^2 \nonumber\\
&=& -1+\frac{j+1}{(n+1)^2}\left(j+(n-j+1)^2\right).\nonumber
\end{eqnarray}
It is then straightforward to calculate that
\[
\text{\textnormal{Var}}^{\text{as}}_n (1) =\sum_{j=0}^n w_j=\frac{n^2(n-1)}{12 (n+1)}
\]
which completes the proof.
\end{proof}

\subsubsection{Weight Degeneracy of Sequential Importance Sampling}\label{wdSIS}

We now turn to the analysis of Sequential Importance Sampling as introduced in Section \ref{secSIStrees} for the present example. We have
\[
\pi_n(j_n)=\left\{\begin{array}{ll}
2^{-j+1}& \text{ for } j<n\\
2^{-n}& \text{ for } j=n\\
\end{array}
 \right.
\]
To prove that depending on the value of $\theta$ the approximation error of $\widetilde{\eta}_n(f)$ may grow exponentially in $n$,  consider the approximation error for the test function $f=1_{\{n_n\}}$.

\begin{cor}\label{SISerrorn}
For $f=1_{\{n_n\}}$ and $\theta > 0$ we have
\[
\mathbb{E}[|\widetilde{\eta}_n(f)-\mu_n(f)|^2]= \left\{\begin{array}{ll}
\frac{2^n-1}{N \left(1+\frac{\theta}{\theta-1}(1-\theta^{-n})\right)^2}& \text{ for }\theta \neq 1,\\
\,&\,\\
\frac{2^n-1}{N(n+1)^2} &\text{ for }\theta = 1
\end{array}\right.
\]
Moreover, $\mathbb{E}[|\widetilde{\eta}_n(f)-\mu_n(f)|^2]$ grows exponentially in $n$ whenever $\theta > 2^{-\frac12}$.
\end{cor}

\begin{proof}[Proof of Corollary \ref{SISerrorn}]
The explicit formula for the error is a direct consequence of Lemma \ref{SISerror}: We have $\pi_n(n_n)=2^{-n}$, and the representation of $\mu_n$ given in Corollary \ref{muktrees} yields
\[
\mu_n(n_n)=\frac{1}{1+\frac{\theta}{\theta-1}(1-\theta^{-n})}
\]
for $\theta \neq 1$ and \[\mu_n(n_n)=\frac{1}{n+1}\] for $\theta=1$. The error grows exponentially in $n$ whenever $2^n\theta^{2n}$ tends to infinity in $n$ which is the case for $\theta> 2^{-\frac12}$.
\end{proof}

Thus Sequential Importance Sampling suffers from weight degeneracy when approximating $f=1_{\{n_n\}}$ even in some cases (i.e. $ 2^{-\frac12}<\theta <1$) where $\mu_n(n_n)$ is decreasing exponentially itself.

\subsubsection{Further Examples}\label{moreexamples}

The fast degeneration of Sequential Importance Sampling in the previous example stems from the fact that the particles' movements only depend on the kernels $K_k$ and do not take into account the reweighting through the functions $g_{k,k+1}$. It is easy to construct a (somewhat artificial) example where this turns out to be an advantage and where accordingly Sequential Importance Sampling outperforms Sequential MCMC. This is done in the following. The notation of the previous example is retained unless otherwise noted.\bigskip

Consider the sequence of state spaces $I_0=\{0_0\}$ and $I_k=\{0_k, 1_k\}$ for $k=1,\ldots, 3$. Define a sequence of probability measures $\mu_k$ on $I_k$ through $\mu_0(0_0)=1$, $\{\mu_1(0_1),\mu_1(1_1)\}=\{\mu_3(0_3),\mu_3(1_3)\}=\{\frac{1}2,\frac{1}2\}$ and $\{\mu_2(0_2),\mu_2(1_2)\} =\{\alpha,1-\alpha\}$ for some  $\alpha \in (0,1)$. The tree structure is given by $p_k(0_{k+1})=0_k$, $p_0(1_1)=0_0$ and, for $k>0$, $p_k(1_{k+1})=1_k$. This implies that 
\[
K_{1}(0_0, 0_1)=K_{1}(0_0, 1_1)=\frac{1}2,
\]
while all other transition kernels are trivial, i.e., for $k>1$ and $j\in \{0,1\}$,  $K_{k}(j_k, j_{k+1})=1$.\bigskip

We first consider the approximation error of Sequential Importance Sampling. The Importance Sampling proposal distribution $\pi_3$  coincides with $\mu_3$. Thus from Lemma \ref{SISerror} we obtain the following: For $f=1_{\{0_3\}}$
\begin{equation}\label{errSIS2}
\mathbb{E}[|\widetilde{\eta}_3(f)-\mu_3(f)|^2]= \frac{\mu_3(0_3)^2}{N}\left( \frac{1}{\pi_3(0_3)}-1\right)  =\frac{1}{4 N}
\end{equation}

Observe that this error is independent of $\alpha$. When moving from $\mu_1$ to $\mu_2$, the weights are changed, but this change is removed when moving (back) to $\mu_3$ and throughout the particles' movements are unaffected. So to say, the particles ``accidentally'' do the right thing when moving from $\mu_0$ to $\mu_1$. To see this, we replace $\mu_1$ by $\mu_1'$ which is essentially the same as $\mu_2$,
$\{\mu_1'(0_1),\mu_1'(1_1)\}=\{\alpha,1-\alpha\}$. Intuitively, this might make the problem easier, because it leads to a ``smoother'' sequence $\mu_k$. Yet the opposite is the case since the proposal distribution $\pi_3'$ is now given by $\pi_3'(0_3)=\alpha$ and $\pi_3'(1_3)=1-\alpha$. Accordingly we get the error bound
\[
\mathbb{E}[|\widetilde{\eta}_3(f)-\mu_3(f)|^2]=\frac{1}{4 N} \left(\frac{1}{\alpha}-1\right),
\]
which gets arbitrarily large for small $\alpha$.\bigskip

Now we consider the asymptotic variance of Sequential MCMC for the same example, again with the original $\mu_1$ and with the test function $f=1_{\{0_3\}}$. We thus have to evaluate
\[
\text{\textnormal{Var}}^{\text{as}}_3(f) =\sum_{j=0}^3 \text{\textnormal{Var}}_{\mu_j}(q_{j,3}(f)).
\]
Using the formula (\ref{qjkTrees1}) for $q_{j,k}(f)$ it is straightforward to calculate that
\begin{eqnarray}
q_{0,3}(1_{\{0_3\}})=\frac12,&\text{    }&q_{1,3}(1_{\{0_3\}})=1_{\{0_1\}},\nonumber\\
q_{2,3}(1_{\{0_3\}})=\frac1{2\alpha}1_{\{0_2\}} &\text{  and  }&q_{3,3}(1_{\{0_3\}})=1_{\{0_3\}}\nonumber.
\end{eqnarray}
Accordingly, we have $\text{\textnormal{Var}}_{\mu_0}(q_{0,3}(f))=0$,
\begin{eqnarray}
\text{\textnormal{Var}}_{\mu_1}(q_{1,3}(f))=\text{\textnormal{Var}}_{\mu_3}(q_{3,3}(f))=\frac14\nonumber
\end{eqnarray}
and
\begin{eqnarray}
\text{\textnormal{Var}}_{\mu_2}(q_{2,3}(f))=\frac14\left( \frac1\alpha-1\right)\nonumber.
\end{eqnarray}
Thus the asymptotic variance is given by
\[
\text{\textnormal{Var}}^{\text{as}}_3(f) =\frac14\left( \frac1\alpha+1\right).
\]
Recall that the asymptotic variance also coincides with the coefficient of the leading term in our error bound of Theorem \ref{thmBound2}. Thus we observe that the approximation error gets arbitrarily large for small values of $\alpha$. This is in contrast to the error (\ref{errSIS2}) of Sequential Importance Sampling for the same example which is independent of $\alpha$.\bigskip

Changing $\mu_1$ to $\mu_1'$ with $\mu_1'(0_1)=\alpha$ and $\mu_1'(1_1)=1-\alpha$ does not lead to a qualitative change of the error bound. We then get
\[q'_{1,3}(1_{\{0_3\}})=\frac1{2\alpha} 1_{\{0_1\}}\;\;\text{   and   }\;\; \text{\textnormal{Var}}_{\mu_1'}(q'_{1,3}(f))= \frac14\left( \frac1\alpha+1\right), \]
which leads to an asymptotic variance of
\[
\text{\textnormal{Var}}^{\text{as}}_3\,'(f) =\frac14\left( \frac2\alpha-1\right).
\]
Observe that again -- despite the fact that the sequence $\mu_0,\mu_1,\mu_2,\mu_3$ varies more strongly than $\mu_0,\mu_1',\mu_2,\mu_3$ -- the asymptotic variance for small values of $\alpha$ is larger under the second sequence than under the first sequence. The reason for this lies in the fact $\mu_1$ is a better approximation of $\mu_3^{\shortrightarrow 1}$ than $\mu_1'$.\bigskip

For $\alpha>\frac12$, the asymptotic variance under $\mu_1'$ is smaller than the one under $\mu_1$ and both are well-behaved. But in this case the asymptotic variance for $f'=1_{\{1_3\}}$ increases more quickly under $\mu_1'$ than under $\mu_1$ as $\alpha$ approaches $1$. In this sense, the asymptotic variance is more stable under $\mu_1$ than under $\mu_1'$.\bigskip

We thus close our comparison of Sequential Importance Sampling and Sequential MCMC on trees with the following conclusion: Sequential Importance Sampling works well if the proposal distribution $\pi_n$ constructed from $\mu_0$ and the transition kernels $K_k$ is sufficiently close to the target distribution $\mu_n$. Sequential MCMC works well if the distributions $\mu_j$ are sufficiently close to the projected distributions $\mu_n^{\shortrightarrow j}$. While there is no obvious relationship between these two properties, it seems clear that Sequential MCMC is more suited to applications where the relative densities $g_{k,k+1}$ play a significant role. Furthermore for both algorithms it is easy to construct examples where they perform arbitrarily bad. Finally note that for the last example we only considered the asymptotic variance of Sequential MCMC. In order to obtain good constants in our error bounds, we also need that $\mu_j$ is sufficiently close to $\mu_k^{\shortrightarrow j}$ for $j<k<n$.

\section{$L_p$-bounds under Local Mixing}\label{LPlocal}

In this section we consider a more standard Sequential MCMC framework where the distributions $\mu_k$ all live on the same state space and where the transition kernels $K_k$ represent (many steps of an) MCMC dynamics with target distributions $\mu_k$. We assume that the MCMC dynamics mix well only \textit{within} the elements of increasingly finer partitions of the state space. These partitions take the role of the tree structure of Section \ref{seqSMCtrees}. Section \ref{modellocal} introduces the setting and connects it to the error bound of Theorem \ref{thmBound}. Section \ref{secStabFKPL} derives stability of the Feynman-Kac propagator $q_{j,k}$. Unlike in Section \ref{seqSMCtrees}, we explicitly take into account the behavior of the dynamics within modes in this section. For this reason, we need two types of additional assumptions: a uniform upper bound on relative densities and sufficiently good mixing within modes.

\subsection{The Model}\label{modellocal}

We return to the setting of Section \ref{Preliminaries} and make a number of additional assumptions: We assume that all the distributions $\mu_k$ live on the same state space $E_k=E$. We assume that the kernel $K_k$ is stationary with respect to $\mu_k$. Accordingly, we assume that the functions $g_{k,k+1}\in B(E)$ are unnormalized relative denisties between $\mu_k$ and $\mu_{k+1}$.\bigskip  

In place of the tree structure of Section \ref{seqSMCtrees} we now introduce a sequence of partitions of $E$. Let $I_0$,...,$I_n$ be a collection of finite index sets. Define $I=I_0 \cup \ldots\cup I_n$ and for $0 \leq k \leq n$
\[
I_{> k}=I_{k+1} \cup \ldots\cup I_n \;\text{  and   } \; I_{< k}=I_0 \cup \ldots\cup I_{k-1}.
\]
For all $j\in I$ there is a set $F_j\in \mathcal{B}(E)$ with $\mu_0(F_j)>0$. Moreover, we assume that for all $0\leq k \leq n$ the collection $(F_j)_{j\in I_k}$ is a disjoint partition of $E$. We assume that partitions successively get finer. For $1\leq k\leq n$, assume that for all $j \in I_k$ there exists an $i\in I_{k-1}$ with $F_j \subseteq F_i$. Thus for $0\leq k\leq n-1$, a well-defined predecessor function $p_k:I_{>k}\rightarrow I_k$ is characterized as follows: For $1\leq k<l\leq n$, $j\in I_k$ and $i \in I_l$ define
\[
p_k(i)=j\;\;\;\; \text{if} \;\;\;F_i \subseteq F_j. 
\]
Conversely, define a successor function $s_k:I_{<k}\rightarrow \mathcal{P}(I_k)$ via
\[
s_k(i)=\{j\in I_k| p_l(j)=i
\}\text{  for  }i \in I_l \text{ with }0\leq l<k.
\]
Thus, for $l<k$ and $i \in I_l$, the collection $(I_j)_{j\in s_k(i)}$ is a disjoint partition of $F_i$. We add the simplifying assumption that particles move between partition elements only through the resampling step.
\begin{ass}\label{assdisc}
For $1 \leq k \leq n$ and $j\in I_k$ let $K_k (1_{F_j})(x)=0$ for all  $ x\in E\setminus F_j$.
\end{ass}
This assumption ensures that if $f$ has support only in $F_j$, $j\in I_k$, then $K_k(f)$ has support only in $F_j$ as well. While this technical assumption will not be literally fulfilled in most applications of interest, it can be seen as an approximation of the fact that particles will move between different modes only extremely rarely through the MCMC dynamics.\bigskip

In order to apply the error bound of Theorem \ref{thmBound} we need to introduce a sequence of norms on $E$. Unlike in the analysis under global mixing in \cite{Schw12} we rely only on local mixing properties. Thus we replace the $L_p$-norms of \cite{Schw12} by stronger norms, which are composed of local $L_p$-norms. To introduce these norms we need a few additional definitions. For $0\leq k \leq n$ and $j\in I$, denote by $\mu_{k,j}\in M_1(E)$ the restriction of $\mu_k$ to $F_j$: For $f\in B(E)$,
\[
\mu_{k,j}(f)=\frac{\mu_k(f1_{F_j})}{\mu_k(F_j)}.
\]
It proves to be convenient to view $\mu_{k,j}$ as a probability distribution on $E$ (and not on $F_j$). Note that we define $\mu_{k,j}$ for all $j\in I$ (and not only for $j\in I_k$). Furthermore, by Assumption \ref{assdisc}, $K_k$ is stationary with respect to $\mu_{k,j}$ for all $j\in I_k$. Now for $0\leq k \leq n$, $j\in I$ and $p \geq 1$, denote by $\|\cdot\|_{k,j,p}$ the $L_p$-norm with respect to $\mu_{k,j}$: For $f\in B(E)$,
\[
\|f\|_{k,j,p}=\mu_{k,j}(|f|^p)^{\frac1p}.
\]
Next define the norm $\|\cdot\|_{k,p}$ to be the maximum over the $L_p$-norms with respect to $\mu_{k,j}$ with $j\in I_k$:
For $f\in B(E)$ and $0\leq k \leq n$,
\[
\|f\|_{k,p}=\max_{j\in I_{k}} \|f\|_{k,j,p}.
\]
With this choice of norm we have
\[
\|f\|_{L_p(\mu_k)} \leq \|f\|_{k,p}.
\]

Now define $\widetilde{c}_{j,k}(p,q)$ to be the constant in an $L_p$-$L_q$-bound for $q_{j,k}$: For $p>  q> 1$ and $0\leq j< k\leq n$ we have
\[
\| q_{j,k}(f)\|_{j,p} \leq
\widetilde{c}_{j,k}(p,q)
 \| f\|_{k,q}\;\; \text{ for all }\;\; f\in B(E).
\]
The following proposition shows how the quantities in the error bound of Theorem \ref{thmBound} can be controlled in terms of the constants $\widetilde{c}_{j,k}(p,q)$. Accordingly, Section \ref{secStabFKPL} is devoted to studying the constants $\widetilde{c}_{j,k}(p,q)$.

\begin{prop}\label{propCon}
Fix $p>2$ and define
\[
c_{j,k}(p)=\max\left(\widetilde{c}_{j,k}\left(p,\frac{p}{2}\right),\;\;\widetilde{c}_{j,k}(2p,p)^2\right).
\]
This choice of $c_{j,k}$ satisfies (\ref{cjninequality}), i.e., for $p>2$, $0\leq j< k\leq n$ and $f\in B(E)$ we have
\[
\max\left(\|1\|_{j,p} \|q_{j,k}(f)^2\|_{j,p},\|q_{j,k}(f)\|_{j,p}^2,\|q_{j,k}(f^2)\|_{j,p} \right)  \leq c_{j,k}(p)\;\|f\|_{k,p}^2.
\]
Moreover
\[
\text{Var}_{\mu_j}(q_{j,k}(f)) \leq  \widetilde{c}_{j,k}(2,2) \|f\|_{k,2}^2
\]
and
\[
\|q_{k,k+1}(1)-1\|_{k,p} \leq \sup_{x\in E} |g_{k,k+1}(x)-1|.
\]
\end{prop}

Upper bounds on the quantities $\widehat{v}_k$ and $\overline{v}_k$ from Theorem \ref{thmBound} follow immediately from the bound on $\text{Var}_{\mu_j}(q_{j,k}(f))$ and from $\|f\|_{k,2} \leq \|f\|_{k,p}$. 

\begin{proof}[Proof of Proposition \ref{propCon}]
We have $\|1\|_{j,p}=1$,
\begin{eqnarray}
\|q_{j,k}(f)\|_{j,p}^2 \leq \|q_{j,k}(f)^2\|_{j,p} =\|q_{j,k}(f)\|_{j,2p}^2 \leq \widetilde{c}_{j,k}(2p,p)^2 \|f\|_{k,p}^2
\nonumber
\end{eqnarray}
and
\[
\|q_{j,k}(f^2)\|_{j,p} \leq \widetilde{c}_{j,k}\left(p,\frac{p}{2}\right) \|f^2\|_{k,\frac{p}{2}}=\widetilde{c}_{j,k}\left(p,\frac{p}{2}\right) \|f\|_{k,p}^2.
\]
This shows that (\ref{cjninequality}) is indeed satisfied with constants $c_{j,k}(p)$. The upper bound on $\text{Var}_{\mu_j}(q_{j,k}(f))$ follows from 
\[
\text{Var}_{\mu_j}(q_{j,k}(f)) \leq \|q_{j,k}(f)\|_{L_2(\mu_j) }^2 \leq \|q_{j,k}(f)\|_{j,2}^2 \leq \widetilde{c}_{j,k}(2,2) \|f\|_{k,2}^2.
\]
The upper bound on $\|q_{k,k+1}(1)-1\|_{k,p}$ follows immediately from $q_{k,k+1}(1)=g_{k,k+1}$ and the definition of $\|\cdot\|_{k,p}$.
\end{proof}

\subsection{Stability of Feynman-Kac Propagators under Local Mixing}\label{secStabFKPL}

We begin with a few more definitions. For $j\in I$, let $m_{k,k+1}(j)$ be the relative change in the mass of $F_j$ between $\mu_k$ and $\mu_{k+1}$,
\[
m_{k,k+1}(j)=\frac{\mu_{k+1}(F_j)}{\mu_{k}(F_j)}.
\]
Furthermore, for $0 \leq k \leq n-1$, denote by $\overline{g}_{k,k+1}$ the normalized relative density between $\mu_k$ and $\mu_{k+1}$,
\[
\overline{g}_{k,k+1}(x)=\frac{g_{k,k+1}(x)}{\mu_k(g_{k,k+1})}, \;\;\;\;\text{ for }\;\; x\in E.
\]
Next we define restricted relative densities: For $0 \leq k \leq n-1$, $j\in I$ and $x\in E$,
\[
\overline{g}_{k,k+1,j}(x)= \frac{1}{m_{k,k+1}(j)}\overline{g}_{k,k+1}(x) \; 1_{F_j}(x).
\]
Observe that with this choice of $\overline{g}_{k,k+1,j}$ we have for $f\in B(E)$, $0 \leq k \leq n-1$ and $j\in I$ that
\[
\mu_{k+1,j}(f)=\frac{\mu_{k+1}(f 1_{F_j})}{\mu_{k+1}(F_j)}
=\frac{1}{m_{k,k+1}(j)}\frac{\mu_{k}(f \overline{g}_{k,k+1} 1_{F_j})}{\mu_{k}(F_j)}
=\mu_{k,j}(\overline{g}_{k,k+1,j} f),
\]
i.e., $\overline{g}_{k,k+1,j}$ is a relative density between $\mu_{k,j}$ and $\mu_{k+1,j}$.\bigskip

We assume a uniform upper bound on restricted relative densities. 
\begin{ass}\label{assbd}
There exists $\gamma > 1$ such that for every $0\leq k \leq n-1$, every $j\in I_k$ and every $x\in F_j$
\[
\overline{g}_{k,k+1,j}(x)= \frac{\mu_{k}(F_j)}{\mu_{k+1}(F_j)}\overline{g}_{k,k+1}(x)  \leq \gamma.
\]
\end{ass}

Note that in the extreme case, where $\overline{g}_{k,k+1}$ is constant on each component $F_j$ with $j\in I_k$, we can choose $\gamma=1$. This extreme case corresponds roughly to what we assumed in Section  \ref{seqSMCtrees}. \bigskip

It proves to be convenient not to work with $q_{j,k}$ directly but to work with  $\hat{q}_{j,k}$ defined as follows: For $1 \leq k \leq n-1$ define  $\hat{q}_{k,k+1}:B(E)\rightarrow B(E)$ by
\[
\hat{q}_{k,k+1}(f)=K_{k}\left(\overline{g}_{k,k+1} f \right). 
\]
Furthermore, for $1 \leq j < k \leq n$ the mapping $\hat{q}_{j,k}:B(E)\rightarrow B(E)$ is given by 
\[
\hat{q}_{j,k}(f)=\hat{q}_{j,j+1}(\hat{q}_{j+1,j+2}(\ldots \hat{q}_{k-1,k}(f)))\;\;\;\text{and}\;\;\;\hat{q}_{k,k}(f)=f.
\]
$q_{j,k}$ and $\hat{q}_{j+1,k}$ are related through
\[
q_{j,k}(f)=\overline{g}_{j,j+1}\hat{q}_{j+1,k}(K_{k}(f)).
\]
In Lemma \ref{lemqhatql} below, we show how $L_p$-$L_q$-bounds for $\hat{q}_{j,k}$ can be used to obtain $L_p$-$L_q$-bounds for $q_{j,k}$.\bigskip

We proceed by considering first $L_2$-bounds for one time-step and then iterated $L_2$-bounds. From these we conclude one-step $L_p$-bounds and then, in Proposition \ref{propLpItL}, iterated $L_p$-bounds for $\hat{q}_{j,k}$. Afterwards, we show how to extend this result to $L_p$-$L_q$-bounds, using local hyperboundedness, and to the original family of operators $q_{j,k}$. Proposition \ref{corGesamtL} concludes the bound for $q_{j,k}$ needed in order to make the constants in the error bound of Theorem \ref{thmBound} explicit.\bigskip

We consider mostly inequalities which bound $\|\hat{q}_{j,k}(f)\|_{j,i,p}$ against $\max_{l\in s_k(i)}\,\|f \|_{k,l,p}$, for $i\in I_j$. The inequalities which bound $\|\hat{q}_{j,k}(f)\|_{j,p}$ against $\,\|f \|_{k,p}$ can then be concluded by taking the maximum over $i\in I_j$. So to say, the latter inequalities are the final results while the former are more useful tools in proving further results.\bigskip 

In order to keep track of how mass is shifted between different components, two more definitions are needed. For $0 \leq j <k\leq n$ and $i\in I_j$ define by $M_{j,k}(i)$ the following iterated version of $m_{j,j+1} (i)$:
\[
M_{j,k}(i) =\max_{l\in s_k(i)} \prod_{r=j}^{k-1}m_{r,r+1}(p_r(l)).
\]
This is the maximal product of relative mass changes one has to go through when moving from $F_i$, $i\in I_j$ to one of its successors $F_l$, $l\in s_k(i)\subseteq I_k$. For the transition from $r$ to $r+1$ the relative mass change of the predecessor of $F_l$ at level $r$ is taken into account. Observe that for $i\in I_j$ we have the relation 
\begin{equation}\label{Mjkrecursion}
M_{j,k}(i)=m_{j,j+1}(i)\max_{l\in s_{j+1}(i)}M_{j+1,k}(l).
\end{equation}
Furthermore, we define for $0 \leq j <k\leq n$ the constant $A_{j,k}$ by
\[
A_{j,k}=\max_{i\in I_j} M_{j,k}(i).
\]

Before we come to local mixing properties and $L_p$-bounds, we briefly look at the $L_1$-case.

\begin{lem}\label{lemL1}
For $0\leq j<k\leq n$, $f\in B(E)$ and $i\in I_j$ we have
\begin{equation}\label{itL1}
\|\hat{q}_{j,k}(f)\|_{j,i,1} \leq M_{j,k}(i) \max_{l\in s_{k}(i)} \|f\|_{k,l,1}.
\end{equation}
Moreover,
\begin{equation}\label{itL1Max}
\|\hat{q}_{j,k}(f)\|_{j,1} \leq A_{j,k} \|f\|_{k,1}.
\end{equation}
\end{lem}
\begin{proof}
We can write
\begin{eqnarray}\label{l1bdstep}
\|\hat{q}_{j,k}(f)\|_{j,i,1} &=&\mu_{j,i}(|K_j(\overline{g}_{j,j+1}\hat{q}_{j+1,k}(f))|)\nonumber\\
&\leq& m_{j,j+1}(i) \mu_{j,i}(\overline{g}_{j,j+1,i}|\hat{q}_{j+1,k}(f)|)\nonumber\\
&\leq&  m_{j,j+1}(i) \max_{l\in s_{j+1}(i)} \mu_{j+1,l}(|\hat{q}_{j+1,k}(f)|)\nonumber\\
&\leq&  m_{j,j+1}(i) \max_{l\in s_{j+1}(i)} \|\hat{q}_{j+1,k}(f)\|_{j+1,l,1}.%\nonumber
\end{eqnarray}
Iterating this bound yields
\begin{eqnarray}
\|\hat{q}_{j,k}(f)\|_{j,i,1}
\leq  m_{j,j+1}(i) \max_{l_{j+1}\in s_{j+1}(i)}  m_{j+1,j+2}(l_{j+1}) \ldots 
\max_{l_{k-1}\in s_{k-1}(l_{k-2})} m_{k-1,k}(l_{k-1}) \|f\|_{k,l_{k-1},1}.\nonumber
\end{eqnarray}
Note that by iterating (\ref{Mjkrecursion}) we obtain 
\begin{eqnarray}
M_{j,k}(i)= m_{j,j+1}(i) \max_{l_{j+1}\in s_{j+1}(i)}  m_{j+1,j+2}(l_{j+1}) \ldots 
\max_{l_{k-1}\in s_{k-1}(l_{k-2})} m_{k-1,k}(l_{k-1}).\nonumber
\end{eqnarray} 
Thus applying 
\[
\|f\|_{k,l_{k-1},1} \leq \max_{l\in s_k(i)} \|f\|_{k,l,1}
\]
in (\ref{l1bdstep}) yields  (\ref{itL1}). Taking the maximum over $i\in I_j$ gives (\ref{itL1Max}).
\end{proof}

The proof illustrates how the constants $M_{j,k}(i)$ and $A_{j,k}$ come into play in our bounds.  The same arguments appear -- in less detail and alongside further complications -- in our proofs for $p>1$. In fact, this didactic purpose is the main motivation behind Lemma \ref{lemL1}. Using Jensen's inequality one can easily show that for all $i\in I$
\[
\|\hat{q}_{j,k}(f)\|_{j,i,1} \leq \frac{\mu_k(F_i)}{\mu_j(F_i)} \|f\|_{k,i,1}.
\]
This implies
\begin{equation}\label{AhutL1}
\|\hat{q}_{j,k}(f)\|_{j,1} \leq \left( \max_{i\in I_j} \frac{\mu_k(F_i)}{\mu_j(F_i)} \right) \|f\|_{k,1},
\end{equation}
which is generally an improvement over Lemma \ref{lemL1}.\bigskip

We now state the local mixing conditions behind our $L_p$-bounds for the case $p\geq 2$.
\begin{ass}\label{asslocmixing}
We have uniform constants $\alpha>0$ and $\beta \in [0,1]$ such that for all $1\leq k < n$, for all $f\in B(E)$ and for all $i\in I_k$
\begin{equation}\label{locmixing}
\|\hat{q}_{k,k+1}(f)\|_{k,i,2}^2 \leq m_{k,k+1}(i)^2\left(\alpha \|f\|_{k+1,i,2}^2+\beta \mu_{k+1,i}(f)^2 \right).
\end{equation}
\end{ass}

One way to ensure that (\ref{locmixing}) holds is to assume that the kernels $K_k$ possess the following contraction property: There exists $\rho\in (0,1)$ such that for all $1\leq k < n$, for all $f\in B(E)$ and for all $i\in I_k$ 
\begin{equation}\label{PoincLocal}
\mu_{k,i}(K_k(f-\mu_{k,i}(f))^2) \leq (1-\rho) \text{\textnormal{Var}}_{\mu_{k,i}}(f).
\end{equation}
Then it can be shown that (\ref{locmixing}) holds with $\alpha=(1-\rho)\gamma$ and $\beta=\rho$. Moreover, (\ref{PoincLocal}) holding with a sufficiently large $\rho$ is equivalent to a local Poincaré inequality with a sufficiently large spectral gap being satisfied, see Section 5.1 of \cite{Schw12} for easily adaptable arguments under global mixing assumptions. Intuitively, a smaller value of $\alpha$ corresponds to better mixing and thus more MCMC steps. Assumption \ref{asslocmixing} immediately implies the following one-step $L_2$-bound for $\hat{q}_{j,k}$.

\begin{cor}
For $1\leq k < n$, for all $f\in B(E)$ and for all $i\in I_k$ we have
\begin{eqnarray}\label{onestepL2}
\|\hat{q}_{k,k+1}(f)\|_{k,i,2}^2 \leq m_{k,k+1}(i)^2 \left(\alpha \left(\max_{l\in s_{k+1}(i)} \|f\|_{k+1,l,2}^2\right)+\beta\left(\max_{l\in s_{k+1}(i)} \mu_{k+1,l}(f)^2 \right)\right).\nonumber\\
\end{eqnarray}
and
\begin{eqnarray}
\|\hat{q}_{k,k+1}(f)\|_{k,2}^2 \leq A_{k,k+1}^2\left(\alpha \|f\|_{k+1,2}^2+\beta \max_{l\in I_{k+1}}  \mu_{k+1,l}(f)^2 \right)
\leq A_{k,k+1}^2(\alpha+\beta) \|f\|_{k+1,2}^2.\nonumber
\end{eqnarray}
\end{cor}

Next we iterate (\ref{onestepL2}) to obtain an $L_2$-bound for more than one step.

\begin{lem}\label{lemL2ItL}
Assume $\alpha<1$. Then for $1\leq j<k \leq n$, $f\in B(E)$ and $i\in I_j$ we have the bounds
\begin{equation}\label{L2it1l}
\|\hat{q}_{j,k}(f)\|_{j,i,2}^2 \leq M_{j,k}(i)^2\left(\alpha^{k-j} \left(  \max_{l\in s_{k}(i)}  \|f\|_{k,l,2}^2 \right) + \frac{\beta}{1-\alpha} \left( \max_{l\in s_{k}(i)} \mu_{k,l}(f)^2 \right)\right),
\end{equation}
and
\begin{equation}\label{L2it12l}
\|\hat{q}_{j,k}(f)\|_{j,i,2} \leq M_{j,k}(i) \frac{1}{(1-\alpha)^{\frac12}} \max_{l\in s_{k}(i)} \|f\|_{k,l,2},
\end{equation}
and
\begin{equation}\label{L2it3l}
\|\hat{q}_{j,k}(f)\|_{j,2} \leq A_{j,k} \frac{1}{(1-\alpha)^{\frac12}}  \|f\|_{k,2}.
\end{equation}
\end{lem}

\begin{proof}
Applying (\ref{onestepL2}) yields 
\begin{eqnarray}\label{firsstep}
 \|\hat{q}_{j,k}(f)\|_{j,i,2}^2 
 \leq m_{j,j+1}(i)^2
  \left(\alpha \max_{l\in s_{j+1}(i)}\|\hat{q}_{j+1,k}(f)\|_{j+1,l,2}^2+\beta \max_{l\in s_{j+1}(i)} \mu_{j+1,l}(\hat{q}_{j+1,k}(f))^2 \right).
 \end{eqnarray}
Arguing as in the proof of Lemma \ref{lemL1} yields the inequality
\begin{equation}\label{L1ittt}
\max_{l\in s_{j+1}(i)} \mu_{j+1,l}(\hat{q}_{j+1,k}(f))^2 \leq M_{j+1,k}(i)^2 \max_{l\in s_{k}(i)} \mu_{k,l}(f)^2,
\end{equation}
which can be used to bound the second term on the right hand side of (\ref{firsstep}). To the first term in (\ref{firsstep}) we can apply again (\ref{onestepL2}) which yields again two terms, one which can be bounded through (\ref{onestepL2}) and one which can be bounded through (\ref{L1ittt}). Iterating this reasoning and collecting the factors $m_{r,r+1}$ into terms $M_{j,k}$ gives us
\begin{equation}\label{L2prg1}
\|\hat{q}_{j,k}(f)\|_{j,i,2}^2 \leq M_{j,k}(i)^2 \left(\alpha^{k-j}  \left(  \max_{l\in s_{k}(i)}  \|f\|_{k,l,2}^2 \right) + \beta\sum_{r=0}^{k-j-1} \alpha^r \max_{l\in s_{k}(i)} \mu_{k,l}(f)^2 \right).
\end{equation}
Applying to this the geometric series inequality yields (\ref{L2it1l}). Since $\mu_{k,l}(f)^2 \leq \|f\|_{k,l,2}^2$ for all $l\in s_{k}(i)$ and since  $\beta \leq 1$, we can conclude from (\ref{L2prg1}) that
\[
\|\hat{q}_{j,k}(f)\|_{j,i,2}^2 \leq M_{j,k}(i)^2  \sum_{r=0}^{k-j} \alpha^r \max_{l\in s_{k}(i)}  \|f\|_{k,l,2}^2,
\]
which implies (\ref{L2it12l}) by the geometric series inequality. Taking the maximum over $i\in I_j$ in (\ref{L2it12l}) gives (\ref{L2it3l}).
\end{proof}

Our next step is the following one-step $L_p$-bound.

\begin{lem}\label{onestepLpLem}
For $1\leq k < n$, for all $f\in B(E)$, for all $i\in I_k$ and for all $p\geq 1$ we have
\begin{eqnarray}\label{onestepLp}
\|\hat{q}_{k,k+1}(f)\|_{k,i,2p}^{2p}\leq m_{k,k+1}(i)^{2p} \gamma^{2p-2} \left(\alpha \left(\max_{l\in s_{k+1}(i)} \|f\|_{k+1,l,2p}^{2p}\right)+\beta \left(\max_{l\in s_{k+1}(i)} \|f\|_{k+1,l,p}^{2p}\right) \right)\nonumber\\
\end{eqnarray}
and
\begin{eqnarray}\label{onestepLpA}
\|\hat{q}_{k,k+1}(f)\|_{k,2p}^{2p} &\leq& A_{k,k+1}^{2p}\gamma^{2p-2}\left(\alpha \|f\|_{k+1,2p}^{2p}+ \beta  \|f\|_{k+1,p}^{2p} \right)\nonumber\\
&\leq& A_{k,k+1}^{2p}\gamma^{2p-2} (\alpha+\beta)  \|f\|_{k+1,2p}^{2p}.
\end{eqnarray}
\end{lem}
\begin{proof}
We can write 
\begin{eqnarray}
\|\hat{q}_{k,k+1}(f)\|_{k,i,2p}^{2p} &=& \mu_{k,i}(|K_k(\overline{g}_{k,k+1}f)|^{2p})\nonumber\\
&\leq& \mu_{k,i}(K_k(\overline{g}_{k,k+1}^p |f|^p )^2)\nonumber\\
&=& \|\hat{q}_{k,k+1}(\overline{g}_{k,k+1}^{p-1} |f|^p )  \|^2_{k,i,2}.\nonumber
\end{eqnarray}
Using (\ref{locmixing}), we can bound this expression  to obtain
\begin{eqnarray}
\|\hat{q}_{k,k+1}(f)\|_{k,i,2p}^{2p} &\leq& m_{k,k+1}(i)^2 \left( \alpha \|\overline{g}_{k,k+1}^{p-1}|f|^p\|^2_{k+1,i,2}
+\beta \|\overline{g}_{k,k+1}^{p-1}|f|^p \|_{k+1,i,1}^2 \right)\nonumber\\
&\leq& m_{k,k+1}(i)^{2p} \gamma^{2p-2} \left( \alpha \| |f|^p\|_{k+1,i,2}^2+\beta \| |f|^p\|_{k+1,i,1}^2\right)\nonumber\\
&\leq& m_{k,k+1}(i)^{2p} \gamma^{2p-2} \left( \alpha \| f \|_{k+1,i,2p}^{2p}+\beta \| f\|_{k+1,i,p}^{2p}\right),\nonumber
\end{eqnarray}
which immediately implies (\ref{onestepLp}). (\ref{onestepLpA}) follows by taking the maximum over $i\in I_k$.
\end{proof}

Next we iterate the bound of Lemma \ref{onestepLpLem} to show how an $L_p$-bound for $\hat{q}_{j,k}$ implies an $L_{2p}$-bound:

\begin{lem}\label{jkstepLpp2p}
Assume that $\alpha \gamma^{2p-2}<1$ and that for some $\delta(p)\geq 1$, for all $1\leq j<k\leq n$, $i\in I_j$ and $f\in B(E)$ the inequality
\begin{equation}\label{IVLpLoc}
\|\hat{q}_{j,k}(f)\|_{j,i,p} \leq M_{j,k}(i) \delta(p) \max_{l\in s_k(i)} \| f \|_{k,l,p}
\end{equation}
is fulfilled. Then we have
\begin{equation}\label{LpLocI}
\|\hat{q}_{j,k}(f)\|_{j,i,2p} \leq M_{j,k}(i) \delta(2p) \max_{l\in s_k(i)} \| f \|_{k,l,2p}
\end{equation}
with
\[
\delta(2p)= \delta(p) \frac{\gamma^{1-\frac1p}}{(1-\alpha \gamma^{2p-2})^{\frac{1}{2p}}}\;.
\]
Moreover, we have
\begin{equation}\label{LpLocII}
\|\hat{q}_{j,k}(f)\|_{j,2p} \leq A_{j,k} \delta(2p)  \| f \|_{k,2p}.
\end{equation}
\end{lem}

\begin{proof}
Define $\theta=\alpha \gamma^{2p-2}$. Iterating the inequality of Lemma \ref{onestepLpLem} and utilizing that $\beta \leq 1$, we get
\begin{eqnarray}\label{firsttttstep}
\|\hat{q}_{j,k}(f)\|_{j,i,2p}^{2p} \leq M_{j,k}(i)^{2p} \theta^{k-j} \left( \max_{l\in s_k(i)} \| f \|_{k,l,2p}^{2p}\right) 
+ \gamma^{2p-2}\sum_{r=j+1}^k \theta^{r-1-j} \max_{l\in s_r(i)} R_{j,r}(l)^{2p} \|\hat{q}_{r,k}(f)\|_{r,l,p}^{2p},\nonumber\\
\end{eqnarray}
where for $l\in I_r$, $R_{j,r}(l)$ is defined by
\[
R_{j,r}(l)=\prod_{t=j}^{r-1} m_{t,t+1}(p_t(l)).
\]
Observe that for $i\in I_j$ we have
\[
M_{j,k}(i)=\max_{l\in s_k(i)} R_{j,k}(l)
\]
and moreover for $j<r<k$ and $i\in I_j$
\begin{equation}\label{MRcombine}
M_{j,k}(i)= \max_{l\in s_r(i)} R_{j,r}(l) M_{r,k} (l)
\end{equation}
Thus applying (\ref{IVLpLoc}) and then (\ref{MRcombine}) to bound the factors $\|\hat{q}_{r,k}(f)\|_{r,l,p}$, we obtain from (\ref{firsttttstep}) the inequality
\begin{eqnarray}
\|\hat{q}_{j,k}(f)\|_{j,i,2p}^{2p} \leq M_{j,k}(i)^{2p} \theta^{k-j} \left( \max_{l\in s_k(i)} \| f \|_{k,l,2p}^{2p}\right)+ \gamma^{2p-2}  M_{j,k}(i)^{2p} \delta(p)^{2p} \sum_{r=j+1}^k \theta^{r-1-j} \max_{l\in s_k(i)} \| f \|_{k,l,p}^{2p}.\nonumber
\end{eqnarray}
Since we assumed $\gamma \geq 1$ and $\delta(p)\geq 1$ and since we have 
\[
\max_{l\in s_k(i)} \| f \|_{k,l,p}^{2p} \leq \max_{l\in s_k(i)} \| f \|_{k,l,2p}^{2p},
\]
we thus have
\[
\|\hat{q}_{j,k}(f)\|_{j,i,2p}^{2p} \leq  \gamma^{2p-2}  M_{j,k}(i)^{2p} \delta(p)^{2p} \left( \max_{l\in s_k(i)} \| f \|_{k,l,p}^{2p}\right) \sum_{r=j+1}^k \theta^{r-1-j} .
\]
By the geometric series inequality and our assumption of $\theta <1$, we thus get
\[
\|\hat{q}_{j,k}(f)\|_{j,i,2p} \leq M_{j,k}(i)^{2p} \delta(2p) \max_{l\in s_k(i)} \| f \|_{k,l,p} , 
\]
with
\[
\delta(2p)= \delta(p) \frac{\gamma^{1-\frac1p}}{(1-\theta)^{\frac{1}{2p}}}.
\]
This shows (\ref{LpLocI}). (\ref{LpLocII}) follows by taking the maximum over $i\in I_j$.
\end{proof}

Combining Lemmas \ref{lemL2ItL} and \ref{jkstepLpp2p} we can now state the key result of this section as follows.

\begin{prop}\label{propLpItL}
For $r\in \mathbb{N}$, consider $p=2^r$ and assume that $\alpha \gamma^{p-2}<1$. Then we have for $1\leq j< k \leq n$, $i\in I_j$ and $f\in B(E)$ the inequality
\begin{equation}\label{LpItLL}
\|\hat{q}_{j,k}(f)\|_{j,i,p} \leq  M_{j,k}(i)\delta(p) \max_{l \in s_k(i)}\|f\|_{k,l,p}, 
\end{equation}
with
\[
\delta(p)=\prod_{j=1}^r \frac{\gamma^{1-2^{-(j-1)}}}{(1-\alpha \gamma^{2^j-2})^{2^{-j}}}< \frac{\gamma^{r-2+2^{-(r-1)}}}{1-\alpha\gamma^{2^r-2}}.
\]
Moreover,
\begin{equation}\label{LpItLLA}
\|\hat{q}_{j,k}(f)\|_{j,p} \leq  A_{j,k}\delta(p) \|f\|_{k,p}. 
\end{equation}
\end{prop}

\begin{proof}
We proceed by induction over $r$. The cases $r=0$ and $r=1$ follow from Lemmas \ref{lemL1} and \ref{lemL2ItL}, respectively. The inequalities for $r > 1$ follow because Lemma \ref{jkstepLpp2p} implies that we can choose
\[
\delta(2^r)=\delta(2) \prod_{j=2}^r \frac{\gamma^{1-2^{-(j-1)}}}{(1-\alpha \gamma^{2^j-2})^{2^{-j}}}.
\]
We can apply Lemma \ref{jkstepLpp2p} iteratively, since $\alpha \gamma^{p-2}<1$ implies $\alpha \gamma^{q-2}<1$, for all $q\leq p$.  For the upper bound on $\delta(p)$, we apply the geometric series equality in the nominator, bound the term in brackets under the exponent in the denominator by $1-\alpha \gamma^{p-2}$ and apply the geometric series inequality to the product. This shows (\ref{LpItLL}). (\ref{LpItLLA}) follows by taking the maximum over $i\in I_j$.
\end{proof}

Since the constants $\delta(2^r)$ are monotonically increasing in $r$, we can immediately extend the bounds of Proposition  \ref{propLpItL} to general $p\geq 1$ using the Riesz-Thorin interpolation theorem (see Davies \cite{DA90}, §1.1.5).

\begin{cor}\label{corLpL}
Consider $p\in [2^r,2^{r+1}]$ for $r\in \mathbb{N}$ and assume $\alpha \gamma^{2^{r+1}-2}<1$. Then for $1\leq j< k \leq n$ and $f\in B(E)$ and $i\in I_j$ we have
\[
\|\hat{q}_{j,k}(f)\|_{j,i,p} \leq  M_{j,k}(i)\delta(p) \max_{l \in s_k(i)}\|f\|_{k,l,p} 
\]
and
\[
\|\hat{q}_{j,k}(f)\|_{j,p} \leq  A_{j,k}\delta(p) \|f\|_{k,p} 
\]
with $\delta(p)=\delta(2^{r+1})$ and $\delta(2^{r+1})$ as defined as in Proposition $\ref{propLpItL}$.
\end{cor}

We still need two more results: one which shows how to translate $L_p$-stability into $L_p$-$L_q$-stability using local hyper-boundedness, and one which relates our bounds for $\hat{q}_{j+1,k}$ to corresponding bounds for  $q_{j,k}$. We first show how $L_p$-$L_q$-inequalities follow from our $L_p$-inequalities and a local hypercontractivity assumption on the kernels $K_j$.

\begin{ass}\label{hypercL}
For $1\leq j< n$, we have a constant $\theta_j(p,q)\geq 0$ such that for all $i\in I_j$ and all $f\in B(E)$ we have
\begin{equation}
\|K_{j}(f)\|_{j,i,p}\leq \theta_j(p,q) \|f\|_{j,i,q}. 
\end{equation}
\end{ass}

Adding this assumption for the remainder of the section, we obtain the following:

\begin{cor}\label{corHypqhatL}
Consider $p\geq 1$ and $q\geq 1$. Let $q\in [2^r,2^{r+1}]$ for $r\in \mathbb{N}$ and assume $\alpha \gamma^{2^{r+1}-2}<1$. Then for $j< k \leq n$, we have
\begin{equation}\label{hypercLI}
\|\hat{q}_{j,k}(f)\|_{j,i,p} \leq M_{j,k}(i) \theta_j(p,q) \gamma^{\frac{q-1}{q}} \delta(q) \max_{l\in s_k(i)} \|f\|_{k,l,q} 
\end{equation}
and
\begin{equation}\label{hypercLII}
\|\hat{q}_{j,k}(f)\|_{j,p} \leq A_{j,k} \theta_j(p,q) \gamma^{\frac{q-1}{q}} \delta(q) \|f\|_{k,q} 
\end{equation}
with $\delta(q)$ as defined in Corollary \ref{corLpL}. 
\end{cor}

\begin{proof}
By Assumption \ref{hypercL} we have
\[
\|\hat{q}_{j,k}(f)\|_{j,i,p}\leq \theta_j(p,q) m_{j,j+1}(i) \max_{l\in s_{j+1}(i)} \|\overline{g}_{j,j+1,i} \hat{q}_{j+1,k}(f)\|_{j,l,q} 
\]
and thus by Corollary $\ref{corLpL}$
\[
\|\hat{q}_{j,k}(f)\|_{j,i,p}\leq \theta_j(p,q) M_{j,k}(i) \gamma^{\frac{q-1}{q}} \delta(q) \max_{l\in s_{k}(i)} \|f\|_{k,l,q}. 
\]
This shows (\ref{hypercLI}). Taking the maximum over $i\in I_j$ proves (\ref{hypercLII}).
\end{proof}

Next, we show how to obtain bounds for $q_{j,k}$ from our bounds for $\hat{q}_{j+1,k}$.

\begin{lem}\label{lemqhatql}
Assume that for some $p\geq 1$ and $q\geq 1$ and for fixed $1\leq j+1< k\leq n$ we have a $\delta\geq 0$ such that for all $l\in I_{j+1}$ and for all $f\in B(E)$ 
\begin{equation}\label{assdeltaM}
\| \hat{q}_{j+1,k}(f) \|_{j+1,l,p} \leq \delta\, M_{j+1,k}(l)\max_{r\in s_k(l)} \| f \|_{k,r,q}.
\end{equation}
Then we have for all $i\in I_j$
\[
\| q_{j,k}(f) \|_{j,i,p} \leq  \gamma^{\frac{p-1}{p}} \delta \,M_{j,k}(i) \,\max_{r\in s_k(i)}  \| f \|_{k,r,q}
\]
and
\[
\| q_{j,k}(f) \|_{j,p} \leq  \gamma^{\frac{p-1}{p}} \delta \,A_{j,k} \,  \| f \|_{k,q}.
\]
\end{lem}

\begin{proof}
Note that we have for $i \in I_j$
\begin{eqnarray}
\| q_{j,k}(f) \|_{j,i,p} &=& \mu_{j,i}(|\overline{g}_{j,j+1}\hat{q}_{j+1,k}(K_k(f))|^p)^{\frac1p}\nonumber\\
&=&  m_{j,j+1}(i)\, \mu_{j,i}(|\overline{g}_{j,j+1,i}\hat{q}_{j+1,k}(K_k(f))|^p)^{\frac1p}\nonumber\\
&\leq& \gamma^{\frac{p-1}{p}}\, m_{j,j+1}(i) \mu_{j+1,i}(|\hat{q}_{j+1,k}(K_k(f))|^p)^{\frac1p}\nonumber\\
&\leq& \gamma^{\frac{p-1}{p}} \, m_{j,j+1}(i) \max_{l\in s_{j+1}(i)} \mu_{j+1,l}(|\hat{q}_{j+1,k}(K_k(f))|^p)^{\frac1p}\nonumber\\
&=& \gamma^{\frac{p-1}{p}} \,  m_{j,j+1}(i) \, \max_{l\in s_{j+1}(i)}  \| \hat{q}_{j+1,k}(K_k(f)) \|_{j,l,p}\nonumber
\end{eqnarray}
and thus by (\ref{assdeltaM})
\begin{eqnarray}
\| q_{j,k}(f) \|_{j,i,p} &\leq& \gamma^{\frac{p-1}{p}} \,  m_{j,j+1}(i) \, \delta\,  \max_{l\in s_{j+1}(i)} M_{j+1,k}(l) \max_{r\in s_{k}(l)}
\| K_k(f) \|_{k,r,q}\nonumber\\
&\stackrel{(\ref{Mjkrecursion})}{\leq}& \gamma^{\frac{p-1}{p}} \,  M_{j,k}(i) \, \delta\, \max_{r\in s_{k}(i)}
\| K_k(f) \|_{k,r,q}\nonumber\\
&\leq& \gamma^{\frac{p-1}{p}} \,  M_{j,k}(i) \, \delta\, \max_{r\in s_{k}(i)}
\| f \|_{k,r,q},\nonumber
\end{eqnarray}
where in the last step we used that by Jensen's inequality $|K_k(f)|^q\leq K_k(|f|^q)$ and that $K_k$ is stationary with respect to $\mu_{k,l}$ for all $l\in I$. This shows the first inequality. The second inequality follows by taking the maximum over $i\in I_j$ on both sides.
\end{proof}

Combining Lemma \ref{lemqhatql} and Corollary \ref{corHypqhatL} we can immediately conclude the types of inequalities needed to derive error bounds from Theorem \ref{thmBound} and Proposition \ref{propCon}:

\begin{prop}\label{corGesamtL}
Consider $p\geq 1$ and $q\geq 1$. Let $q\in [2^r,2^{r+1}]$ for $r\in \mathbb{N}$ and assume $\alpha \gamma^{2^{r+1}-2}<1$. Then for $1 \leq j< k \leq n$ we have
\begin{equation}\label{hypercLIIq}
\|q_{j,k}(f)\|_{j,p} \leq \widetilde{c}_{j,k}(p,q) \|f\|_{k,q} 
\end{equation}
with
\[
\widetilde{c}_{j,k}(p,q) =  A_{j,k} \theta_j(p,q)  \gamma^{\frac{p-1}{p}} \gamma^{\frac{q-1}{q}} \delta(q),
\]
where $\delta(q)$ is as defined in Corollary \ref{corLpL}. 
\end{prop}

The stability inequalities in this section differ from those in \cite{Schw12} by containing the factor $A_{j,k}$ on the right-hand side. For the case treated in that paper, i.e., $|I_j|=1$ for $0\leq j\leq n$ and $F_i=E$ for all $i \in I$, we obtain $A_{j,k}=1$. Thus the results of the present section contain those obtained in there as special cases. For the case of invariant partitions treated in Eberle and Marinelli \cite{EM09}, i.e., $|I_j|=|I_k|$ for $0\leq j<k\leq n$, we obtain
\begin{equation}\label{AjKInvPart}
A_{j,k}=\max_{i\in I_j}\frac{\mu_k(F_i)}{\mu_j(F_i)},
\end{equation}
which is a discrete-time analogue of their constant.\bigskip

Compared to the setting of Section \ref{seqSMCtrees}, the present setting is more general in two respects: We take into account local mixing and local variations in relative densities. To disentangle these two factors to some extent, consider the case where we take into account local mixing but assume that the relative densities $g_{k,k+1}$ are constant on each of the sets $F_j$ with $j\in I_k$. In that case, we have $\gamma=1$ and the inequality (\ref{hypercLIIq}) in Proposition \ref{corGesamtL} becomes
\[
\|q_{j,k}(f)\|_{j,p} \leq A_{j,k}  \theta(p,q) \frac{1}{1-\alpha} \|f\|_{k,q}. 
\]
Thus, compared to the results of Section \ref{seqSMCtrees}, we obtain different norms, we obtain the constants $A_{j,k}$ which are similar to but greater than the corresponding constants $\sqrt{d_{j,k}}$ in Section \ref{seqSMCtrees} and we obtain an additional factor taking into account hyperboundedness and local mixing.

\section{Outlook}\label{MMrellDiscussion}

Since the analysis of Section  \ref{seqSMCtrees} abstracts from most of the local structure -- and thus from some possible problems --  it should be seen as providing a rough but intuitive criterion for identifying settings where the algorithm works or does not work. Consider for instance the mean field Potts model for which slow mixing of the Parallel Tempering algorithm was proved by Bhatnagar and Randall \cite{BR04}, see their paper for more details about the model. Basically, in this model there is a distribution $\mu_0$ which is unimodal and a distribution $\mu_n$ which has four modes of roughly equal weight. Along the transition from $\mu_0$ to $\mu_n$, three additional modes arise which are immediately well-separated from the initial one and have a tiny initial mass, say $\varepsilon$. Then we obtain huge constants $d_{j,k}$ proportional to $\varepsilon^{-2}$ in Proposition \ref{propDJK} and thus in the error bound.\bigskip

For the mean field Ising model, Madras and Zheng \cite{MZ02} proved rapid mixing of Parallel Tempering. The results of Section  \ref{seqSMCtrees} suggest that the same should be true for Sequential MCMC: In the mean field Ising model there is basically one mode which is split into two modes of equal weight at some point as we move from $\mu_0$ to $\mu_n$. In this case, each mode has exactly the same weight as its successors and thus $d_{j,k}=1$. Therefore we can expect a good performance of the algorithm.\bigskip

A similar good performance of Sequential MCMC can be expected in the problem of estimating the parameters of mixture distributions as described in Celeux, Hurn and Robert \cite{CHR98}. There, the target distribution $\mu_n$ is a distribution on the parameter space $\mathbb{R}^{l\times k}$ with the symmetry property that $\mu_n$ assigns the same weight to all states $\theta'$ which can be obtained by permutating the rows of a given $\theta \in \mathbb{R}^{l\times k}$. To illustrate the source of this multimodality problem consider the following example: The Gaussian mixture distributions
\[
0.25\, \mathcal{N}(5,1)+0.75\,  \mathcal{N}(0,1) \text{ and } 0.75\,  \mathcal{N}(0,1)+0.25\,  \mathcal{N}(5,1) 
\]
are identical, i.e., some permutations of the parameters correspond to the same mixture distribution. The target distribution $\mu_n$ of MCMC is a posterior distribution on the parameter space and thus assigns the same weight to $\theta=(0.25,5,1;0.75,0,1)\in \mathbb{R}^{2\times 3}$ and $\theta'=(0.75,0,1; 0.25,5,1)\in \mathbb{R}^{2\times 3}$. The results of Section  \ref{seqSMCtrees} suggest that if this type of permutation symmetry is the sole source of multimodality, Sequential MCMC should work well, since the symmetry is retained when tempering the target distribution. Thus, the areas around each local mode have the same weight at all ``temperatures'', $d_{j,k}=1$. This intuition is confirmed by simulations of Celeux, Hurn and Robert \cite{CHR98} who study an example along these lines and demonstrate that Simulated Tempering can move between local modes while simple MCMC cannot. Permutation symmetries are also one source of multimodality in models from chemical physics, see Wales \cite{W03}. Another message of the analysis of Section  \ref{seqSMCtrees} is however that multimodality caused by permutation symmetries is one of the easiest to deal with cases of multimodality. Thus, examples of this type are rather limited toy examples for testing a multilevel MCMC algorithm's ability to move between disconnected modes.\bigskip

The results of Section  \ref{LPlocal} convey basically the same intuition as those of Section  \ref{seqSMCtrees} but they explicitly take into account local mixing and sufficient similarity of $\mu_k$ and $\mu_{k+1}$ \textit{within} disconnected components. Both aspects are important to keep in mind: Assume we choose $n=1$, let $\mu_0$ be a distribution for which we have excellent global mixing and let $\mu_1$ be an arbitrary other distribution which is strongly multimodal. This setting can be projected to a tree where a number of leafs branch from a single root. Then the results of Section  \ref{seqSMCtrees} seem to suggest, that Sequential MCMC with only these two distributions should work very well, since the root of the tree has mass $1$ under $\mu_0$ and its successors have mass $1$ under $\mu_1$. This -- obviously false -- conclusion can be drawn from Section  \ref{seqSMCtrees} since the model does not take into account local variations in relative densities which may lead to huge errors in the Importance Sampling Resampling step. Basically, in Section \ref{seqSMCtrees} we make the -- implicit -- assumption that relative densities are constant within each component. Similar ``wrong intuitions'' can be derived from the model of Section  \ref{seqSMCtrees} by disregarding the fact that local mixing has to be guaranteed within each component.\bigskip

To sum up, the results of Section  \ref{seqSMCtrees} are to be seen as convenient tools for developing intuitive results about the algorithm such as our comparison of Resampling and Weighting in Section \ref{anexample}. The results of Section \ref{LPlocal} are to be seen as first steps toward explicit error bounds for Sequential MCMC in actual examples. They are only first steps for two reasons: 1) Assumption \ref{assdisc} will not literally be fulfilled in most applications since the probability of transitions between disconnected modes will be negligible but positive. 2) While the Mixing Assumptions \ref{asslocmixing} and \ref{hypercL} are weaker than the corresponding assumptions in most of literature, see \cite{Schw11,Schw12} for more discussion, they will be hard to check in most applications of interest. Concretely, for diffusion processes hyperboundedness and a spectral gap follow, e.g., from a Logarithmic Sobolev inequality which can be verified using the Bakry-Éméry criterion, see, e.g., \cite{ABCFGMRS00}. For the usual discrete-time MCMC dynamics on non-compact state spaces similar tools for proving hyperboundedness are still missing.\bigskip 

Our results demonstrate that problems of the algorithm which stem from disconnected components gaining mass can generally not be alleviated by increasing the number of interpolating distributions: Adding additional steps in the sequence $\mu_0,\ldots, \mu_n$ can only increase the constant $d_{j,k}$ and $A_{j,k}$. This separates this type of problem from problems associated with large local variations in relative densities, i.e. a large value of $\gamma$ in Assumption \ref{assbd}, which can be controlled fairly well through the number of interpolating distributions. The only way to control the constants $d_{j,k}$ and $A_{j,k}$ seems to be to choose an entirely different sequence $\mu_0,\ldots, \mu_{n-1}$.\bigskip

Finally, the results of Sections \ref{seqSMCtrees} and \ref{LPlocal} suggest that, generally, a bad performance of Sequential MCMC is not a property of the target distribution $\mu_n$ but a property of the approximating sequence $\mu_0,\ldots,\mu_{n-1}$ which is a parameter in the algorithm, not in the problem of interest. So far, the theoretical literature on multilevel MCMC algorithms, i.e., Simulated Tempering, Parallel Tempering and Sequential MCMC has largely focused on flattening a target distribution by tempering. In the applied literature, there are many more, sometimes model-specific, proposals for choosing a sequence of distributions such as cutting off the Hamiltonian at chosen minimum levels in addition to tempering, varying the system size or spatial coarse graining, see, e.g. \cite{KZW06,LS98, LYZ06}. A more systematic study of methods for approximating target distributions seems to be an important and highly challenging task for future research.


\begin{thebibliography}{99}

\bibitem{ABCFGMRS00} C. Ané, S. Blachère, D. Chafaï, P. Fougères, I. Gentil, F. Malrieu, C. Roberto and G. Scheffer, \textit{Sur les inégalités de Sobolev logarithmiques}, Panoramas et
Synthèses, 10, Société Mathématique de France, Paris, 2000.

\bibitem{BR04}
N. Bhatnagar and D. Randall, \textit{Torpid Mixing of Simulated Tempering on the Potts Model}, Proceedings of the 15th ACM-SIAM Symposium on Discrete Algorithms (SODA), 278-287, 2004.

\bibitem{CMR05} O. Cappé, E. Moulines and T. Rydén, \textit{Inference in Hidden Markov Models}, Springer, New York, 2005.

\bibitem{CPS92}
S. Caracciolo, A. Pelissetto and A. D. Sokal, \textit{Two Remarks on Simulated Tempering}, Unpublished Manuscript, 1992.


\bibitem{CHR98}
G. Celeux, M. Hurn and C. P. Robert, \textit{Computational and inferential difficulties with mixture posterior distributions}, Journal of the American Statistical Association, 95, 957-970, 2000.

\bibitem{CDG11}
F. C\'erou, P. Del Moral and A. Guyader,\textit{ A nonasymptotic variance theorem for unnormalized Feynman-Kac particle models}, Annales de l'Institut Henri Poincar\'e, Probabilit\'es et Statistiques, 47, 629-649, 2011. 


\bibitem{C04} N. Chopin, \textit{Central Limit Theorem for Sequential Monte Carlo methods and its application to Bayesian inference}, Annals of Statistics, 32, 2385-2411, 2004.


\bibitem{DA90} E. B. Davies, \textit{Heat kernels and spectral theory}, Cambridge University Press, Cambridge, 1990.


\bibitem{DM96}
P. Del Moral, \textit{Nonlinear filtering: interacting particle solution}, Markov Processes and Related Fields, 2, 555-579, 1996.

\bibitem{DM05}
P. Del Moral, \textit{Feynman-Kac Formulae}, Springer, New York, 2004.

\bibitem{DDJ05} P. Del Moral, A. Doucet and A. Jasra, \textit{Sequential Monte Carlo Samplers}, Journal of the Royal Statistical Society B, 68, 411-436, 2006.


\bibitem{DM00} P. Del Moral and L. Miclo, \textit{Branching and interacting particle systems approximations of Feynman-Kac formulae with applications to nonlinear filtering}, Séminaire de Probabilités XXXIV, Lecture Notes in Mathematics, vol. 1729, Springer, Berlin, 2000, pp. 1-145. 



\bibitem{DM08} R. Douc and E. Moulines, \textit{Limit theorems for weighted samples with applications to Sequential Monte Carlo Methods}, Annals of Statistics, 36, 2344-2376, 2008.


\bibitem{EM09} A. Eberle and C. Marinelli, \textit{$L^p$ estimates for Feynman-Kac propagators with time-dependent reference measures}, Journal of Mathematical Analysis and Applications, 365, 120-134,  2010.

\bibitem{EM08} A. Eberle and C. Marinelli, \textit{Quantitative approximations of evolving probability measures and sequential Markov Chain Monte Carlo methods}, Probability Theory and Related Fields, forthcoming, 2012.




\bibitem{GSS93} N. Gordon, D. Salmond and A. Smith, \textit{Novel approach to nonlinear/non-Gaussian Bayesian state estimation}, IEE Proceedings F Radar and Signal Processing, 140, 107-113, 1993.

\bibitem{JSTV04}
M. Jerrum, J.-B. Son, P. Tetali, and E. Vigoda, \textit{Elementary bounds on Poincaré and log-Sobolev constants for decomposable Markov chains}, Annals of Applied Probability, 14, 1741-1765, 2004. 


\bibitem{KZW06}
S. C. Kou, Q. Zhou and W. H. Wong, \textit{Equi-energy sampler with applications in statistical inference and statistical mechanics}, Annals of Statistics, 34, 1581-1619, 2006. 


\bibitem{K05}
H. R. K\"unsch, \textit{Recursive Monte-Carlo filters: algorithms and theoretical analysis}, Annals of Statistics, 33, 1983-2021, 2005.


\bibitem{LS98}
J. S. Liu and C. Sabatti, \textit{Simulated sintering: Markov chain
Monte Carlo with spaces of varying dimensions}, In: Bayesian Statistics,
J. M. Bernardo, J. O. Berger, A. P.  Dawid and A. F. M. Smith (eds.), Oxford University Press, New York, 1998.

\bibitem{LYZ06}
E. Lyman, F. M. Ytreberg and D. M. Zuckerman, \textit{Resolution Exchange Simulation}, Physical Review Letters, 96, 028105, 2006. 


\bibitem{MR02}
N. Madras and D. Randall, \textit{Markov chain decomposition for convergence rate analysis}, Annals of Applied Probability, 12, 581-606, 2002.

\bibitem{MZ02}
N. Madras and Z. Zheng, \textit{On the swapping algorithm}, Random Structures and Algorithms, 22, 66-97, 2002. 


\bibitem{N01} R. M. Neal, \textit{Annealed importance sampling}, Statistics and Computing, 11, 125-139, 2001.


\bibitem{Schw11}
N. Schweizer, \textit{Non-asymptotic Error Bounds for Sequential MCMC Methods.}, Doctoral Thesis, University of Bonn, 2011.

\bibitem{Schw12}
N. Schweizer, \textit{Non-asymptotic Error Bounds for Sequential MCMC and Stability of Feynman-Kac Propagators.}, Arxiv preprint 	1204.2382, 2012.


\bibitem{W03} D. J.\ Wales, \textit{Energy Landscapes}, Cambridge University Press,
Cambridge, 2003.

\bibitem{W11}
N. Whiteley, \textit{Sequential Monte Carlo samplers: error bounds and insensitivity to initial conditions}, Arxiv preprint 1103.3970, 2011.

\bibitem{DSH09}
D. B. Woodard, S. C. Schmidler and M. L. Huber, \textit{Conditions for rapid mixing of parallel and simulated tempering on multimodal distributions},  Annals of Applied Probability, 19, 617-640, 2009a.


\bibitem{DSH09b}
D. B. Woodard, S. C. Schmidler and M. L. Huber, \textit{Sufficient conditions for torpid mixing of parallel and simulated tempering} Electronic Journal of Probability, 14, 780-804, 2009b.

\end{thebibliography}
\end{document}